\documentclass[conference]{IEEEtran}

\usepackage{preamble}

\usepackage{cite}
\usepackage{amsmath,amssymb,amsfonts}
\usepackage{algorithmic}
\usepackage{textcomp}
\usepackage{xcolor}
\usepackage{pifont}

\def\BibTeX{{\rm B\kern-.05em{\sc i\kern-.025em b}\kern-.08em
    T\kern-.1667em\lower.7ex\hbox{E}\kern-.125emX}}

\newcommand*{\Cu}{C_{\mathrm{u}}}
\newcommand*{\Cl}{C_{\mathrm{l}}}

\newcommand{\group}[1]{\mathcal{#1}}

\newcommand*{\symgrp}[1]{\mathcal{S}_{#1}}
\newcommand*{\Sn}{\symgrp{n}}

\renewcommand*{\l}{\ell}

\newcommand*{\dRNN}{\dR_{\geq 0}}

\newcommand{\fL}[1]{f_{\mathrm{L},#1}}
\newcommand{\fR}[1]{f_{\mathrm{R},#1}}
\newcommand{\tfL}[1]{\tilde{f}_{\mathrm{L},#1}}
\newcommand{\tfR}[1]{\tilde{f}_{\mathrm{R},#1}}
\newcommand{\hfL}[1]{\hat{f}_{\mathrm{L},#1}}
\newcommand{\hfR}[1]{\hat{f}_{\mathrm{R},#1}}

\newcommand{\hzero}{\hat 0}

\newcommand{\tr}{\mathsf{T}}

\newcommand{\hvx}{\vectvect{\hat{x}}}

\newcolumntype{C}[1]{>{\centering\arraybackslash$}p{#1}<{$}}

\usepackage{environ}
\NewEnviron{pbalign}
  {
    \begingroup
    \allowdisplaybreaks
    \begin{align}
      \BODY
    \end{align}
    \endgroup
  }
\NewEnviron{pbalign*}
  {
    \begingroup
    \allowdisplaybreaks
    \begin{align*}
      \BODY
    \end{align*}
    \endgroup
  }

\newcommand{\Footnotetext}[2]{\begin{figure}[!b]\footnotesize%
  \vspace{-3ex}\hrulefill\hfill\makebox[0em]{}\hfill\makebox[0em]{}%
  \par${}^{#1}$ #2\vspace{-0.60ex}\end{figure}\addtocounter{figure}{0}}

\begin{document}

\title{Double-Cover-Based Analysis of the \\
       Bethe Permanent of Non-negative Matrices}

\author{
  \IEEEauthorblockN{Kit Shing NG}
  \IEEEauthorblockA{
    \textit{Department of Information Engineering} \\
    \textit{The Chinese University of Hong Kong} \\
    Shatin, N.T., Hong Kong \\
    nks020@ie.cuhk.edu.hk \\[-0.5cm]
  }
  \and
  \IEEEauthorblockN{Pascal O. Vontobel}
  \IEEEauthorblockA{
    \textit{Department of Information Engineering} \\
    \textit{The Chinese University of Hong Kong} \\
    Shatin, N.T., Hong Kong \\
    pascal.vontobel@ieee.org \\[-0.5cm]
  }
}

\maketitle

\begin{abstract}
  The permanent of a non-negative matrix appears naturally in many information
  processing scenarios. Because of the intractability of the permanent beyond
  small matrices, various approximation techniques have been developed in the
  past. In this paper, we study the Bethe approximation of the permanent and
  add to the body of literature showing that this approximation is very well
  behaved in many respects. Our main technical tool are topological double
  covers of the normal factor graph whose partition function equals the
  permanent of interest, along with a transformation of these double covers.

\end{abstract}

\Footnotetext{}{The work described in this paper was partially supported by
  grants from the Research Grants Council of the Hong Kong Special
  Administrative Region, China (Project Nos.\ CUHK 14209317 and CUHK
  14208319).}

\section{Introduction}
\label{sec:introduction:1}

Let $n$ be a positive integer. Recall that the permanent of a matrix
$\matr{A} = (a_{i, j}) \in \dR^{n \times n}$ is defined to be (see, e.g.,
\cite{Minc:78})
\begin{align}
  \perm(\matr{\matr{A}})
    &\defeq 
       \sum_{\sig \in \setset{S}_n}
         \prod_{i \in [n]} a_{i, \sig(i)},
           \label{eq:perm}
\end{align}
where $\setset{S}_n$ is the symmetric group of degree $n$, i.e., the group of
all the $n!$ permutations of $[n] \defeq \{ 1, \ldots, n \}$. It is well known
that exactly computing the permanent is suspected to be a hard problem in
general (see, e.g., the discussion in~\cite{6352911}).

In this paper we focus on the particularly important special case when
$\matr{A}$ is a non-negative matrix, i.e., when the entries of the matrix
$\matr{A}$ take on non-negative values. Various approaches have been proposed
to efficiently numerically approximate the permanent of such matrices (see,
e.g., the discussion in~\cite{6352911}). One of these approximations is the
so-called Bethe permanent $\perm_{\Bethe}(\matr{A})$, which is based on the
Bethe approximation from statistical physics~\cite{Yedidia:Freeman:Weiss:05:1}
and is given as the solution of some optimization problem derived from
$\matr{A}$.

In contrast to the original definition of $\perm_{\Bethe}(\matr{A})$, which is
given in terms of an optimization problem, one can, using the techniques that
were developed in~\cite{6570731}, give a combinatorial characterization of
$\perm_{\Bethe}(\matr{A})$. Namely,
\begin{align}
  \perm_{\Bethe}(\matr{A})
    &= \limsup_{M \to \infty} \ 
         \perm_{\Bethe,M}(\matr{A}),
           \label{eq:ZBethe:perm:intro:1} \\
  \perm_{\Bethe,M}(\matr{A})
    &\defeq
       \sqrt[M]{\Big\langle \!
                  \perm(\matr{A}^{\! \uparrow \matr{\tilde P}})
                \! \Big\rangle_{\matr{\tilde P} \in \tilde \Phi_M}}.
                             \label{eq:ZBethe:perm:intro:2}
\end{align}
Here the expression under the root sign represents the (arithmetic) average of
$\perm(\matr{A}^{\! \uparrow \matr{\tilde P}})$ over all $M$-covers of
$\matr{A}$, $M \! \geq \! 1$. (See the upcoming sections for the technical
details.)

Note that we can write
\begin{align}
  \underbrace{
    \frac{\perm(\matr{A})}
         {\perm_{\Bethe}(\matr{A})}
  }_{\text{\ding{192}}}
    &= \underbrace{
         \frac{\perm(\matr{A})}
              {\perm_{\Bethe,2}(\matr{A})}
       }_{\text{\ding{193}}}
       \cdot
       \underbrace{
         \frac{\perm_{\Bethe,2}(\matr{A})}
              {\perm_{\Bethe}(\matr{A})}
       }_{\text{\ding{194}}} \ .
         \label{eq:Z:ratios:1}
\end{align}
Numerically computing the ratios in~\eqref{eq:Z:ratios:1} for various choices
of matrices $\matr{A}$ shows that a significant contribution to the
ratio~\ding{192} comes from the ratio~\ding{193}. Therefore, understanding the
ratio~\ding{193} can give useful insights to understanding the
ratio~\ding{192}. The central topic of this paper is to make this observation
mathematically more precise.

Let $\matr{A}$ be an arbitrary non-negative matrix of size $n \times n$. The
key technical result of this paper is that
\begin{align}
  \frac{\perm_{\Bethe,2}(\matr{A})}
       {\perm(\matr{A})}
    &= \sqrt
       {
         \sum_{\sig_1, \sig_2 \in \Sn}
           p(\sig_1) 
           \cdot
           p(\sig_2)
           \cdot
           2^{-c(\sig_1,\sig_2)}
       },
         \label{eq:key:result:1}
\end{align}
where
$p(\sigma) \defeq \big( \prod_{i \in [n]} a_{i, \sig(i)} \big) /
\perm(\matr{A})$ is the probability mass function on $\setset{S}_n$ induced by
$\matr{A}$ and where $c(\sig_1,\sig_2)$ is the number of cycles of length
larger than one in the cycle notation expression of the permutation
$\sig_1 \circ \sig_2^{-1}$.

This result is then leveraged to make the following ana\-lytical statements
related to~\eqref{eq:Z:ratios:1}:
\begin{itemize}

\item If $\matr{A} = \matr{1}_{n \times n}$, i.e., the all-one matrix of size
  $n \times n$, then\footnote{The notation $a(n) \sim b(n)$ stands for
    $\lim\limits_{n \to \infty} \frac{a(n)}{b(n)} = 1$.}
  \begin{align}
    \frac{\perm(\matr{A})}
         {\perm_{\Bethe}(\matr{A})}
      &\sim
         \sqrt{\frac{2\pi n}{\e}},
    \quad
    \frac{\perm(\matr{A})}
         {\perm_{\Bethe,2}(\matr{A})}
       \sim
         \sqrt[4]{\frac{\pi n}{\e}}.
           \label{eq:permanent:ratio:all:one:matrix:1}
  \end{align}
  Observe that, up to a factor $\sqrt{2}$, the ratios~\ding{193}
  and~\ding{194} in~\eqref{eq:Z:ratios:1} are the same for this matrix!

\newcommand{\Stringone}{$\perm_{\Bethe,2}(\matr{A})$}
\newcommand{\Stringtwo}{$\perm_{\Bethe}(\matr{A})$}
\newcommand{\Stringthr}{$\perm(\matr{A})$}

\begin{figure}
  \vskip-0.30cm
  \begin{center}
  \resizebox{1.05\columnwidth}{!}{%
  \input{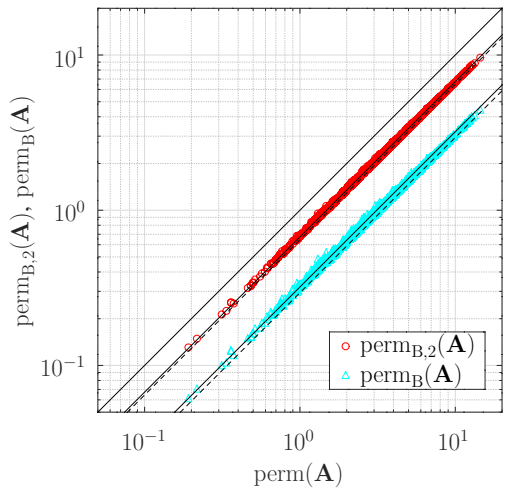}
  }
  \end{center}
  \caption{Simulation results discussed in Section~\ref{sec:introduction:1}.}
  \label{fig:simulation:results:1}
  \vspace{-0.50cm}
\end{figure}

\item If $\matr{A}$ is a random matrix of size $n \times n$ whose entries are
  i.i.d.\ according to some distribution with support over the non-negative
  reals, then
  \begin{align}
    \hspace{-0.25cm}
    \gamma_{\Bethe,2}(n)
      &\defeq
         \frac{\sqrt{\expvalbig{\perm(\matr{A})^2}}}
              {\sqrt{\expvalbig{\perm_{\Bethe,2}(\matr{A})^2}}}
       \sim
         \gamma'_{\Bethe,2}(n)
       \defeq
         \sqrt[4]{\frac{\pi n}{\e}}.
           \label{eq:permanent:ratio:random:matrix:1}
  \end{align}
  Interestingly, although the numerator and the denominator on the left-hand
  side of the above expression both \emph{depend} on the chosen distribution,
  the right-hand side of the above expression is \emph{independent} of this
  distribution! 

\end{itemize}

Empirically, it actually appears that much stronger results than
in~\eqref{eq:permanent:ratio:random:matrix:1} can potentially be
proven. Namely, the plot in Fig.~\ref{fig:simulation:results:1} shows the
following:
\begin{itemize}

\item For $n = 5$, we randomly generated $1000$ matrices $\matr{A}$ of size
  $n \times n$, where the entries are i.i.d.\ according to the uniform
  distribution in the interval $[0,1]$. For each matrix $\matr{A}$, we plotted
  a red circle at the location
  $\bigl( \perm(\matr{A}), \perm_{\Bethe,2}(\matr{A}) \bigr)$ and a cyan
  triangle at the location
  $\bigl( \perm(\matr{A}), \perm_{\Bethe}(\matr{A}) \bigr)$.

\item Let $\gamma_{\Bethe,2}(n)$ be the quantity defined
  in~\eqref{eq:permanent:ratio:random:matrix:1}. The solid black line going
  through the cluster of red circles is the location of pairs
  $\bigl( \perm(\matr{A}), \perm_{\Bethe,2}(\matr{A}) \bigr)$ for which
  $\perm(\matr{A}) / \perm_{\Bethe,2}(\matr{A}) = \gamma_{\Bethe,2}(n)$.

\item Let $\gamma'_{\Bethe,2}(n)$ be the quantity defined
  in~\eqref{eq:permanent:ratio:random:matrix:1}. The dashed black line going
  through the cluster of red circles is the location of pairs
  $\bigl( \perm(\matr{A}), \perm_{\Bethe,2}(\matr{A}) \bigr)$ for which
  $\perm(\matr{A}) / \perm_{\Bethe,2}(\matr{A}) = \gamma'_{\Bethe,2}(n)$.

\item The solid and dashed black lines going through the cluster of cyan
  triangles are similar to the solid and dashed red lines going through the
  cluster of red circles, respectively. They are based on our conjecture that
  the value of the ratio $\perm(\matr{A}) / \perm_{\Bethe}(\matr{A})$ follows
  closely the value of the ratio
  $\perm(\matr{1}_{n \times n}) / \perm_{\Bethe}(\matr{1}_{n \times n})$.

\end{itemize}

Motivated by Fig.~\ref{fig:simulation:results:1}, we leave it to future
research to make stronger analytical statements than
in~\eqref{eq:permanent:ratio:random:matrix:1} w.r.t.\ the distribution of the
ratio $\perm(\matr{A}) / \perm_{\Bethe,2}(\matr{A})$, about the distribution
of the ratio $\perm(\matr{A}) / \perm_{\Bethe}(\matr{A})$, etc. Such results
will contribute toward rigorously justifying the empirical observation of the
usefulness of the Bethe approximation of the permanent (see, e.g.,
\cite{Chertkov:Kroc:Vergassola:08:1,Huang:Jebara:09:1,Watanabe:Chertkov:10:1,
  6283654,6804280,Williams:Lau:14:1,JMLR:v14:chertkov13a,7908949}). Some of
these papers replace the (usually intractable) optimization problem
$\arg \max_{\matr{A} \in \setset{A}} \perm(\matr{A})$ by the more tractable
optimization problem
$\arg \max_{\matr{A} \in \setset{A}} \perm_{\Bethe}(\matr{A})$, where
$\setset{A}$ is a set of matrices of interest. For this approximation to work
well, the value of the ratio $\perm(\matr{A}) / \perm_{\Bethe}(\matr{A})$ is
irrelevant as long as it is nearly the same for all matrices
$\matr{A} \in \setset{A}$.

It is worthwhile to mention that various results for the ratio
$\perm(\matr{A}) / \perm_{\Bethe}(\matr{A})$ have been developed in the
past. In particular, lower and upper bounds on this ratio can be found
in~\cite{Gurvits:11:2,Gurvits:Samorodnitsky:14:1,8948682,Straszak:Vishnoi:19:1}
for arbitrary non-negative matrices $\matr{A}$ of size $n \times n$ and
in~\cite{pmlr-v134-anari21a} for non-negative matrices $\matr{A}$ of size
$n \times n$ and of a given non-negative matrix rank. While rather
non-trivial, these results are not strong enough / not suitable to derive the
results mentioned above.

This paper is structured as follows: in Section~\ref{sec:nfg:1} we discuss a
family of normal factor graphs whose partition function equals the permanent
of a non-negative matrix $\matr{A}$. Afterwards, we apply a technique
from~\cite{7746637} for analyzing $\perm_{\Bethe,2}(\matr{A})$: in
Section~\ref{sec:double:cover:based:analysis:1} for general non-negative
matrices, in Section~\ref{sec:all:one:matrix:1} for all-one matrices, and in
Section~\ref{sec:iid:matrices:1} for random non-negative matrices with i.i.d.\
entries. We conclude the paper in Section~\ref{sec:conclusion:1}. Finally, we
have collected many of the proofs in the appendices.

The main tool of this paper are topological graph covers. Note that graph
covers have also been used in other contexts toward understanding and
quantifying the Bethe approximation of various quantities of interest. For
example, graph covers were used to analyze so-called log-supermodular
graphical models~\cite{Ruozzi:12:1,7746637} and weighted homomorphism counting
problems over bipartite graphs~\cite{Csikvari:Ruozzi:Shams:22:1}.

\section{Normal Factor Graph Representation}
\label{sec:nfg:1}

\begin{figure}
  \hspace{0.65cm}
  \begin{subfigure}{0.45\linewidth}
    \includegraphics[width=0.85\linewidth]{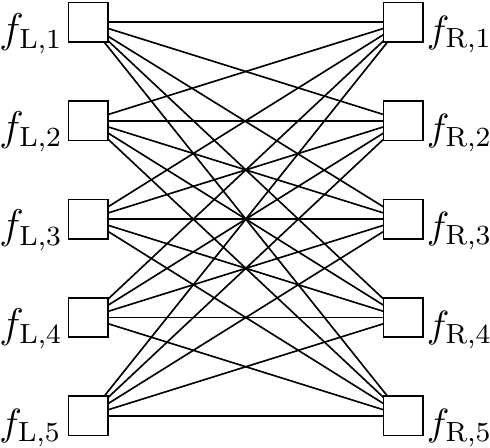}
  \end{subfigure}
  \begin{subfigure}{0.45\linewidth}
    \includegraphics[width=0.85\linewidth]{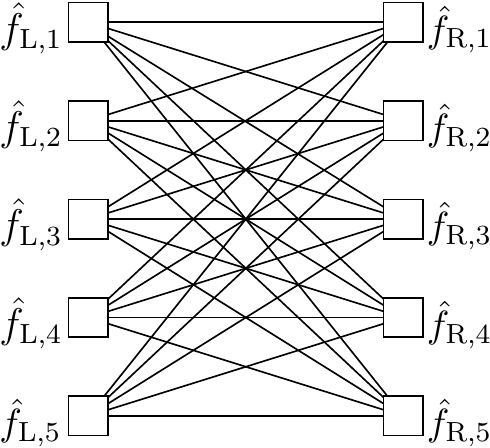}
  \end{subfigure}
  \vspace{0.25cm}
  \caption{NFGs used in Sections~\ref{sec:nfg:1} and
    \ref{sec:double:cover:based:analysis:1}. (Here, $n = 5$.)}
  \label{fig:ffg:permanent:1}
  \vspace{-0.5cm}
\end{figure}

We assume that the reader is familiar with the basics of factor
graphs~\cite{Kschischang:Frey:Loeliger:01,Forney:01:1,Loeliger:04:1}, in
particular with a variant of factor graphs called normal factor graphs (NFGs).
Recall that an NFG is a graph consisting of vertices (called function nodes)
and edges, where so-called local functions are associated with vertices and
where variables are associated with edges. The local functions are such that
their arguments are only variables associated with edges incident on the
corresponding factor node. The global function is then defined to be the
product of the local functions, and the partition function is defined to be
the sum of the global function over all possible assignments to the variables
associated with the edges.

Let $\matr{A} = (a_{i,j})$ be an arbitrary non-negative matrix of size
$n \times n$. There are various ways of defining an NFG such that its
partition function equals $\perm(\matr{A})$. In this paper, we use the NFG
$\graphN(\matr{A})$ in Fig.~\ref{fig:ffg:permanent:1}~(left), which is the
same as in~\cite{6352911}:
\begin{itemize}

\item The NFG $\graphN(\matr{A})$ is based on a complete bipartite graph with
  two times $n$ vertices.

\item For every $i,j \in [n]$, let the variable associated with the
  edge connecting $f_{\mathrm{L},i}$ with $f_{\mathrm{R},j}$ be called
  $x_{i,j}$ and take value in the set $\setset{X} \defeq \{ 0, 1 \}$.

\item For every $i \in [n]$, let
  \begin{align*}
    \!\!\!\!\!
    f_{\mathrm{L},i}(x_{i,1}, \ldots, x_{i,n})
      &\defeq
         \begin{cases}\!
           \sqrt{a_{i,j}} 
             & \!\!\!\!
               \begin{array}{c}
                 \text{$\exists j \in [n]$ s.t. 
                       $x_{i,j} = 1$;} \\
                 \text{$x_{i,j'} = 0$, 
                       $\forall j' \in [n] \setminus \{ j \}$}
               \end{array} \\[0.50cm]
           0 & \text{(otherwise)}
         \end{cases}
  \end{align*}
  For every $j \in [n]$, let
  \begin{align*}
    \!\!\!\!\!
    f_{\mathrm{R},j}(x_{1,j}, \ldots, x_{n,j})
      &\defeq
         \begin{cases}\!
           \sqrt{a_{i,j}} 
             & \!\!\!\!
               \begin{array}{c}
                 \text{$\exists i \in [n]$ s.t. 
                       $x_{i,j} = 1$;} \\
                 \text{$x_{i',j} = 0$, 
                       $\forall i' \in [n] \setminus \{ i \}$}
               \end{array} \\[0.50cm]
           0 & \text{(otherwise)}
         \end{cases}
  \end{align*}

\item The global function is defined to be
  \begin{align*}
    g(x_{1,1}, x_{1,2}, \ldots, x_{n,n})
      &\defeq
         \left(
           \prod_{i \in [n]}
             f_{\mathrm{L},i}(x_{i,1}, \ldots, x_{i,n})
         \right) \\
      &\hspace{0.25cm}
         \cdot
         \left(
           \prod_{j \in [n]}
             f_{\mathrm{R},j}(x_{1,j}, \ldots, x_{n,j})
         \right).
  \end{align*}

\item The partition function (or partition sum) is defined to be
  \begin{align*}
    Z(\graph{N})
      &\defeq
         \sum_{x_{1,1}, x_{1,2}, \ldots, x_{n,n}}
           g(x_{1,1}, x_{1,2}, \ldots, x_{n,n}).
  \end{align*}

\end{itemize}
One can verify the following:
\begin{itemize}

\item $g(x_{1,1}, x_{1,2}, \ldots, x_{n,n}) = \prod_{i \in [n]} a_{i,\sig(i)}$
  if there exists a permutation $\sig \in \setset{S}_n$ such that for all
  $i, j \in [n]$ either $x_{i,j} = 1$ if $j = \sig(i)$ or $x_{i,j} = 0$
  otherwise.

\item $g(x_{1,1}, x_{1,2}, \ldots, x_{n,n}) = 0$ if there exists no such
  permutation $\sig \in \setset{S}_n$.

\item $Z(\graph{N}) = \perm(\matr{A})$.

\end{itemize}

For NFGs whose local functions take on non-negative values, the
paper~\cite{Yedidia:Freeman:Weiss:05:1} introduced the Bethe approximation
$Z_{\Bethe}(\graphN)$ of the partition function $Z(\graphN)$ as the solution
of some optimization problem derived from~$\graphN$. In~\cite{6570731} it was
then shown that $Z_{\Bethe}(\graphN)$ has the following combinatorial
characterization:
\begin{align}
  Z_{\Bethe}(\graphN)
    &= \limsup_{M \to \infty} \ 
         Z_{\Bethe,M}(\graphN),
           \label{eq:ZBethe:intro:1} \\
  Z_{\Bethe,M}(\graphN)
    &\defeq
       \sqrt[M]{\Big\langle \!
                  Z(\graph{\tilde N})
                \! \Big\rangle_{\graph{\tilde N} \in \setset{\tilde N}_{M}}},
                             \label{eq:ZBethe:intro:2}
\end{align}
where the expression under the root sign represents the (arithmetic) average
of $Z(\graph{\tilde N})$ over all $M$-covers $\graph{\tilde N}$ of $\graphN$,
$M \geq 1$. For the details of the definition of (topological) graph covers,
we refer to~\cite{6570731}.

Interestingly, in the context of the NFG $\graphN(\matr{A})$, the expressions
in~\eqref{eq:ZBethe:intro:2} turns into the expression
in~\eqref{eq:ZBethe:perm:intro:2}, where for positive integers $M$ we have
defined
\begin{align*}
  \matr{A}^{\! \uparrow \matr{\tilde P}}
    &\defeq
       \begin{pmatrix}
         a_{1,1} \matr{\tilde P}^{(1,1)}
           & \cdots
           & a_{1,n} \matr{\tilde P}^{(1,n)} \\
         \vdots                   
           &
           &  \vdots \\
         a_{n,1} \matr{\tilde P}^{(n,1)} 
           & \cdots
           & a_{n,n} \matr{\tilde P}^{(n,n)}
       \end{pmatrix}, \\
  \tilde \Phi_M
    &\defeq
       \left\{
          \matr{\tilde P} 
            = \big\{ \matr{\tilde P}^{(i,j)} \big\}_{i \in [n], j \in [n]}
          \ \middle| \ 
          \matr{\tilde P}^{(i,j)} \in \setset{P}_{M \times M}
       \right\}, \\
  \setset{P}_{M \times M}
    &\defeq
       \text{(set of permutation matrices of size $M \times M$)}.
\end{align*}
Note that the matrix $\matr{A}^{\! \uparrow \matr{\tilde P}}$ has size
$(Mn) \times (Mn)$. (For more details, see the discussion in
\cite[Section~VI]{6570731}.)

In the rest of this paper, we will analyze double covers of
$\graphN(\matr{A})$, i.e., graph covers of $\graphN(\matr{A})$ for $M = 2$.

\section{Double-Cover-Based Analysis: General Case}
\label{sec:double:cover:based:analysis:1}

The paper~\cite{7746637} introduced a technique for analyzing double covers of
an arbitrary NFG $\graphN$. We refer the interested reader to~\cite{7746637}
for all the technical details. Here we just state the main result when applied
to the NFG $\graphN(\matr{A})$.

Namely, using \cite[Theorem~4]{7746637}, one obtains the following result.
(Note that on the right-hand side of~\eqref{eq:double:cover:1} only a single
NFG appears, which is in contrast to the right-hand side
of~\eqref{eq:ZBethe:intro:2} that averages over multiple NFGs.)

\begin{proposition}
  \label{prop:average:double:cover:1}

  Let $\matr{A}$ be a non-negative matrix of size $n \times n$. It holds that
  \begin{align}
    \perm_{\Bethe,2}(\matr{A})
      &= \sqrt{Z\bigl( \hat \graphN(\matr{A}) \bigr)},
           \label{eq:double:cover:1}
  \end{align}
  where the NFG $\hat \graphN(\matr{A})$ is defined as follows (see
  Fig.~\ref{fig:ffg:permanent:1}~(right)):
  \begin{itemize}
  
  \item The NFG $\hat \graphN(\matr{A})$ is based on a complete bipartite graph
    with two times $n$ vertices, i.e., the same graph underlying
    $\graphN(\matr{A})$.
  
  \item For every $i,j \in [n]$, let the variable associated with the edge
    connecting $\hat f_{\mathrm{L},i}$ with $\hat f_{\mathrm{R},j}$ be called
    $\hat x_{i,j}$ and take value in the set
    \begin{align*}
      \setset{\hat X} 
        &\defeq
           \setset{X} \times \setset{X}
         = \bigl\{ (0,0), \ (0,1), \ (1,0), \ (1,1) \bigr\}.
    \end{align*}
  
  \item For every $i \in [n]$, let
    $\hat f_{\mathrm{L},i}(\hat x_{i,1}, \ldots, \hat x_{i,n}) \defeq$
    \begin{align*}
        &
           \begin{cases}
             a_{i,j}
               & \begin{array}{c}
                   \text{$\exists j \in [n]$ s.t. 
                         $\hat x_{i,j} = (1,1)$;} \\
                   \text{$\hat x_{i,j'} = (0,0)$, 
                         $\forall j' \in [n] \setminus \{ j \}$}
                 \end{array} \\[0.50cm]
             \sqrt{a_{i,j} a_{i,j'}}
               & \begin{array}{c}
                   \text{$\exists j, j' \in [n]$, $j \neq j'$, s.t.} \\
                   \text{$\hat x_{i,j} = (0,1)$, $\hat x_{i,j'} = (0,1)$;} \\
                   \text{$\hat x_{i,j''} = (0,0)$, 
                         $\forall j'' \in [n] \setminus \{ j, j' \}$}
                 \end{array} \\[0.75cm]
             0 & \text{(otherwise)}
           \end{cases}
    \end{align*}

  \item For every $j \in [n]$, let
    $\hat f_{\mathrm{R},j}(\hat x_{1,j}, \ldots, \hat x_{n,j}) \defeq$
    \begin{align*}
        &
           \begin{cases}
             a_{i,j}
               & \begin{array}{c}
                   \text{$\exists i \in [n]$ s.t. 
                         $\hat x_{i,j} = (1,1)$;} \\
                   \text{$\hat x_{i',j} = (0,0)$, 
                         $\forall i' \in [n] \setminus \{ i \}$}
                 \end{array} \\[0.50cm]
             \sqrt{a_{i,j} a_{i',j}}
               & \begin{array}{c}
                   \text{$\exists i, i' \in [n]$, $i \neq i'$, s.t.} \\
                   \text{$\hat x_{i,j} = (0,1)$, $\hat x_{i',j} = (0,1)$;} \\
                   \text{$\hat x_{i'',j} = (0,0)$, 
                         $\forall i'' \in [n] \setminus \{ i, i' \}$}
                 \end{array} \\[0.75cm]
             0 & \text{(otherwise)}
           \end{cases}
    \end{align*}
  
  \item The global function is defined to be
    \begin{align*}
      \hat g(\hat x_{1,1}, \hat x_{1,2}, \ldots, \hat x_{n,n})
        &\defeq
           \left(
             \prod_{i \in [n]}
               \hat f_{\mathrm{L},i}(\hat x_{i,1}, \ldots, \hat x_{i,n})
           \right) \\
        &\hspace{0.25cm}
           \cdot
           \left(
             \prod_{j \in [n]}
               \hat f_{\mathrm{R},j}(\hat x_{1,j}, \ldots, \hat x_{n,j})
           \right).
    \end{align*}
  
  \item The partition function (or partition sum) is defined to be
    \begin{align*}
      Z(\graph{\hat N})
        &\defeq
           \sum_{\hat x_{1,1}, \hat x_{1,2}, \ldots, \hat x_{n,n}}
             \hat g(\hat x_{1,1}, \hat x_{1,2}, \ldots, \hat x_{n,n}).
    \end{align*}
  
  \end{itemize}
\end{proposition}

\begin{proof}
  See Appendix~\ref{sec:proof:prop:average:double:cover:1}
\end{proof}

Using a mapping from $\setset{S}_n \times \setset{S}_n$ to the set of valid
configurations of $Z(\graph{\hat N})$, i.e., the set of configurations
$(\hat x_{1,1}, \hat x_{1,2}, \ldots, \hat x_{n,n})$ that yield
$\hat g(\hat x_{1,1}, \hat x_{1,2}, \ldots, \hat x_{n,n}) \neq 0$, allows us
to reformulate Proposition~\ref{prop:average:double:cover:1} as follows.

\begin{proposition}
  \label{prop:average:double:cover:2}

  Let $\matr{A}$ be a non-negative matrix of size $n \times n$. It holds that
  \begin{align}
    \frac{\perm_{\Bethe,2}(\matr{A})}
         {\perm(\matr{A})}
      &= \sqrt
         {
           \sum_{\sig_1, \sig_2 \in \Sn}
             p(\sig_1) 
             \cdot
             p(\sig_2)
             \cdot
             2^{-c(\sig_1,\sig_2)}
         },
  \end{align}
  where
  \begin{align}
    p(\sigma)
      &\defeq
         \frac{\ \
               \prod\limits_{i \in [n]} 
                 a_{i,\sig(i)} 
               \ \ 
              }
              {\perm(\matr{A})}
  \end{align}
  is the probability mass function on $\setset{S}_n$ induced by $\matr{A}$ and
  where $c(\sig_1,\sig_2)$ is the number of cycles of length larger than one
  in the cycle notation expression of the permutation
  $\sig_1 \circ \sig_2^{-1}$, or, equivalently, where $c(\sig_1,\sig_2)$ is
  the number of orbits of length larger than one of the permutation
  $\sig_1 \circ \sig_2^{-1}$.
\end{proposition}

\begin{proof}
  See Appendix~\ref{sec:proof:prop:average:double:cover:2}.
\end{proof}

In the following, we evaluate the expression in
Proposition~\ref{prop:average:double:cover:2} for different setups.

\section{Double-Cover-Based Analysis: \\
              All-One Matrix}
\label{sec:all:one:matrix:1}

We have the following result for the all-one matrix of size $n \times
n$. (Note that~\eqref{eq:double:cover:Bethe:perm:all:one:matrix:1:1} was
already proven in~\cite[Lemma~48]{6570731}.)

\begin{theorem}
  \label{theorem:all:one:matrix:1}

  Let $\matr{A} \defeq \matr{1}_{n \times n}$, i.e., the all-one matrix of size
  $n \times n$. It holds that
  \begin{align}
    \frac{\perm(\matr{A})}
         {\perm_{\Bethe}(\matr{A})}
      &\sim
         \sqrt{\frac{2\pi n}{\e}},
           \label{eq:double:cover:Bethe:perm:all:one:matrix:1:1} \\
    \frac{\perm(\matr{A})}
         {\perm_{\Bethe,2}(\matr{A})}
      &\sim
         \sqrt[4]{\frac{\pi n}{\e}},
           \label{eq:double:cover:Bethe:perm:all:one:matrix:1:2} \\
    \frac{\perm_{\Bethe,2}(\matr{A})}
         {\perm_{\Bethe}(\matr{A})}
      &\sim
         \sqrt{2} \cdot \sqrt[4]{\frac{\pi n}{\e}}.
           \label{eq:double:cover:Bethe:perm:all:one:matrix:1:3}
  \end{align}
\end{theorem}

We see that understanding the ratio
$\frac{\perm(\matr{A})}{\perm_{\Bethe,2}(\matr{A})}$ goes a long way toward
understanding the ratio $\frac{\perm(\matr{A})}{\perm_{\Bethe}(\matr{A})}$.

\begin{proof}
  A sketch of the proof of this theorem is given below. For all the details of
  the proof, see Appendix~\ref{sec:proof:theorem:all:one:matrix:1}.
\end{proof}

Let us sketch the proof of Theorem~\ref{theorem:all:one:matrix:1}. First,
thanks to the special structure of the all-one matrix,
Proposition~\ref{prop:average:double:cover:2} simplifies as follows.

\begin{corollary}[of Proposition~\ref{prop:average:double:cover:2}]
  \label{cor:prop:average:double:cover:2:1}

  Let $\matr{A} = \matr{1}_{n \times n}$. It holds that
  \begin{align}
    \frac{\perm_{\Bethe,2}(\matr{A})}
         {\perm(\matr{A})}
      &= \sqrt
           {
             \frac{1}{n!}
             \sum_{\sig \in \setset{S}_n}
               2^{-c(\sig)}
           },
  \end{align}
  where $c(\sig)$ is the number of cycles of length larger than one in the
  cycle notation expression of the permutation $\sig$, or, equivalently, where
  $c(\sig)$ is the number of orbits of length larger than one of the
  permutation $\sig$.
\end{corollary}

\begin{proof}
  This follows from $p(\sig) = 1/n!$ for all $\sig \in \setset{S}_n$ and
  simplifying the summation.
\end{proof}

Using results for the cycle index of the symmetric group (see, e.g.,
\cite{Biggs:89:1,Wilf2005}), one can analyze the expression
\begin{align*}
  Z_n
    &\defeq 
       \frac{1}{n!}
       \sum_{\sig \in \setset{S}_n}
         2^{-c(\sig)}
     = \frac{1}{|\setset{S}_n|}
       \sum_{\sig \in \setset{S}_n}
         2^{-c(\sig)}
\end{align*}
for $n \in \dN$. Indeed, one obtains the following results.
\begin{itemize}

\item Let $Z_0 \defeq 1$. Then for $n \in \dN$ it holds that
  \begin{align*}
    Z_n
      &= \frac{1}{n}
           \cdot
           \left(
             Z_{n-1}
             +
             \frac{1}{2}
               \cdot
               \sum_{\ell=2}^{n}
                 Z_{n-\ell}
           \right).
  \end{align*}

\item Analyzing the expression $n \cdot Z_n - (n \! - \! 1) \cdot Z_{n-1}$
  leads to the following simplified expression for $n \geq 2$:
  \begin{align*}
    Z_n
      &= Z_{n-1}
         -
         \frac{Z_{n-2}}{2n}.
  \end{align*}
  This expression is convenient for numerically evaluating $Z_n$, however, it
  does not appear that there are closed-form expressions for $Z_n$.

\item Nevertheless, using mathematical induction, one can obtain the following
  lower and upper bounds which hold for all positive integers $n$:
  \begin{align*}
    \frac{1}{\sqrt{2}}
    \cdot
    \frac{1}{\sqrt{n}}
      &\leq
         Z_n
       \leq
         \frac{3}{2 \sqrt{2}}
         \cdot
         \frac{1}{\sqrt{n}}.
  \end{align*}

\item In order to obtain an even more precise characterization of
  $Z_n$, one can use
  \begin{align*}
    Z_n
      &= \frac{1}{n!}
           \cdot
           B_n
             \left(
               0! \cdot 1, 
               1! \cdot \frac{1}{2}, 
               \dots, 
               (n \! - \! 1)! \cdot \frac{1}{2}
             \right),
  \end{align*}
  where $B_n$ denotes the $n$-th complete exponential Bell polynomial (see,
  e.g., \cite{Comtet1974}). Then, carefully bounding the above expression, one
  obtains
  \begin{align*}
    Z_n
      &\sim
         \sqrt{\frac{\e}
                    {\pi n}
              }.
  \end{align*}
  This expression, via Corollary~\ref{cor:prop:average:double:cover:2:1},
  immediately yields~\eqref{eq:double:cover:Bethe:perm:all:one:matrix:1:2}.

\item Finally, Eq.~\ref{eq:double:cover:Bethe:perm:all:one:matrix:1:3} is then
  obtained by combining~\eqref{eq:double:cover:Bethe:perm:all:one:matrix:1:1}
  and~\eqref{eq:double:cover:Bethe:perm:all:one:matrix:1:2}.

\end{itemize}

\section{Double-Cover-Based Analysis: I.I.D.\ Matrix}
\label{sec:iid:matrices:1}

In this section we consider matrices satisfying the following assumption.

\begin{assumption}
  \label{ass:iid:matrices:1}

  Let $\matr{A}$ be a random matrix of size $n \times n$ whose entries are
  i.i.d.\ according to some distribution with support over the non-negative
  reals, first moment $\mu_1$, second moment~$\mu_2$, and, consequently,
  variance $\mu_2 - \mu_1^2$.
\end{assumption}

\begin{theorem}
  \label{theorem:iid:matrices:1}

  Given Assumption~\ref{ass:iid:matrices:1}, it holds that
  \begin{align}
    \frac{\sqrt{\expval{\perm(\matr{A})^2}}}
         {\sqrt{\expval{\perm_{\Bethe,2}(\matr{A})^2}}}
       \sim
         \sqrt[4]{\frac{\pi n}{\e}}.
           \label{eq:double:cover:Bethe:perm:iid:matrices:1:1}
  \end{align}
\end{theorem}

Interestingly, although $\expval{\perm_{\Bethe,2}(\matr{A})^2}$ and
$\expval{\perm(\matr{A})^2}$ both \emph{depend} on the chosen distribution
(see the calculations below), the right-hand side
of~\eqref{eq:double:cover:Bethe:perm:iid:matrices:1:1} is \emph{independent}
of this distribution!

\begin{proof}
  A sketch of the proof of this theorem is given below. For all the details of
  the proof, see Appendix~\ref{sec:proof:theorem:iid:matrices:1}.
\end{proof}

Let us sketch the proof of Theorem~\ref{theorem:iid:matrices:1}. We start by
evaluating $\expval{\perm_{\Bethe,2}(\matr{A})^2}$. From
Proposition~\ref{prop:average:double:cover:2} we obtain the following
corollary.

\begin{corollary}[of Proposition~\ref{prop:average:double:cover:2}]
  \label{cor:prop:average:double:cover:2:2}

  Given Assumption~\ref{ass:iid:matrices:1}, it holds that
  \begin{align*}
    \expval{\perm_{\Bethe,2}(\matr{A})^2}
      &= n!
         \cdot
         \sum_{\sig \in \setset{S}}
           \mu_2^{c_1(\sig)}
           \cdot
           \left(
             \prod_{\ell \geq 2} \mu_1^{2 \cdot \ell \cdot c_{\ell}(\sig)}
           \right)
           \cdot
           2^{-c(\sig)}
  \end{align*}
  where for $\ell \geq 1$ we define $c_{\ell}(\sig)$ to be the number of
  cycles of length $\ell$ in the cycle notation expression of the
  permutation~$\sig$. Note that $\sum_{\ell \geq 1} \ell \cdot c_{\ell} = n$
  and $\sum_{\ell \geq 2} c_{\ell}(\sig) = c(\sig)$.
\end{corollary}

\begin{proof}
  From Proposition~\ref{prop:average:double:cover:2} it follows that
  \begin{align*}
    &
    \hspace{-0.05cm}
    \expval{\perm_{\Bethe,2}(\matr{A})^2} \\
      &= \!\!\!\!\!
         \sum_{\sig_1, \sig_2 \in \Sn}
         \!\!\!\!\!
           \expval{
             \left(
               \prod_{i_1 \in [n]} \!\!
                 a_{i_1,\sig_1(i_1)}
             \right)
             \! \cdot \!
             \left(
               \prod_{i_2 \in [n]} \!\!
                 a_{i_2,\sig_2(i_2)}
             \right)
           }
           \!\! \cdot \!
           2^{-c(\sig_1,\sig_2)} \\
      &= \sum_{\sig_1, \sig_2 \in \Sn}
           \mu_2^{c_1(\sig_1,\sig_2)}
           \cdot
           \left(
             \prod_{\ell \geq 2} \mu_1^{2 \cdot \ell \cdot c_{\ell}(\sig_1,\sig_2)}
           \right)
           \cdot
           2^{-c(\sig_1,\sig_2)} \\
      &= n!
         \cdot
         \sum_{\sig \in \setset{S}}
           \mu_2^{c_1(\sig)}
           \cdot
           \left(
             \prod_{\ell \geq 2} \mu_1^{2 \cdot \ell \cdot c_{\ell}(\sig)}
           \right)
           \cdot
           2^{-c(\sig)},
  \end{align*}
  where for $\ell \geq 1$ we define $c_{\ell}(\sig_1,\sig_2)$ to be the number
  of cycles of length $\ell$ in the cycle notation expression of the
  permutation $\sig_1 \circ \sig_2^{-1}$.
\end{proof}

Using results for the cycle index of the symmetric group, one can analyze the
expression
\begin{align*}
  Z_n
   &\defeq 
      \frac{1}{n!}
      \sum_{\sig \in \setset{S}_n}
        \mu_2^{c_1(\sig)}
           \cdot
           \left(
             \prod_{\ell \geq 2} \mu_1^{2 \cdot \ell \cdot c_{\ell}(\sig)}
           \right)
           \cdot
           2^{-c(\sig)} \\
   &= \frac{1}{|\setset{S}_n|}
      \sum_{\sig \in \setset{S}_n}
        \mu_2^{c_1(\sig)}
           \cdot
           \left(
             \prod_{\ell \geq 2} \mu_1^{2 \cdot \ell \cdot c_{\ell}(\sig)}
           \right)
           \cdot
           2^{-c(\sig)},
\end{align*}
for positive integers $n$. In order to evaluate $Z_n$, one can use
\begin{align*}
  Z_n
    &= \frac{1}{n!}
         \cdot
         B_n
           \left(
             0! \cdot \mu_2, 
             1! \cdot \frac{1}{2} (\mu_1^2)^2, 
             \dots, 
             (n \! - \! 1)! \cdot \frac{1}{2} (\mu_1^2)^n
           \right)\!,
\end{align*}
where $B_n$ denotes the $n$-th complete exponential Bell polynomial. Then,
carefully bounding the above expression, one obtains
\begin{align*}
  Z_n
    &\sim
       \mu_1^{2n}
       \cdot
       \e^{(\mu_2/\mu_1^2)-1}
       \cdot
       \sqrt
       {\frac{\e}
             {\pi n}
       },
\end{align*}
The combination of the above results yields
\begin{align*}
  \expval{\perm_{\Bethe,2}(\matr{A})^2}
    &\sim
       (n!)^2
       \cdot
       \mu_1^{2n}
       \cdot
       \e^{(\mu_2/\mu_1^2)-1}
       \cdot
       \sqrt
       {\frac{\e}
             {\pi n}
       }.
\end{align*}

Turning our attention now to $\expval{\perm(\matr{A})^2}$, calculations
similar to the above calculations yield
\begin{align*}
  \expval{\perm(\matr{A})^2}
    &= n!
       \cdot
       \sum_{\sig \in \setset{S}}
         \mu_2^{c_1(\sig)}
         \cdot
         \left(
           \prod_{\ell \geq 2} (\mu_1^2)^{\ell c_{\ell}(\sig)}
         \right) \\
    &= \Psi_n(\mu_2, \mu_1^2, 1) \\
    &\sim
       (n!)^2 
         \cdot
         \mu_1^{2n} 
         \cdot
         \e^{(\mu_2/\mu_1^2)-1}.
\end{align*}

Finally, combining the above results, we
get~\eqref{eq:double:cover:Bethe:perm:iid:matrices:1:1}.

\section{Conclusion}
\label{sec:conclusion:1}

We conclude this paper with a few remarks:
\begin{itemize}

\item In this paper we have studied the ratios
  $\frac{\perm(\matr{A})} {\perm_{\Bethe}(\matr{A})}$ and
  $\frac{\perm(\matr{A})} {\perm_{\Bethe,2}(\matr{A})}$. While it is more
  desirable to characterize the former, the latter seems to be more
  tractable. In particular, the latter can be used to obtain qualitative
  \emph{and} quantitative insights into the former. We leave it to future
  research to strengthen the results that are presented in this paper w.r.t.\
  these ratios.

\item For the all-one matrices studied in Sections~\ref{sec:all:one:matrix:1},
  some further investigations show that the summation
  in~\eqref{eq:key:result:1} is dominated by permutations
  $\sig_1, \sig_2 \in \setset{S}_n$ for which
  $c(\sig_1,\sig_2) = \Theta\bigl( \ln(n) \bigr)$, ultimately leading to the
  result
  $\frac{\perm(\matr{A})}{\perm_{\Bethe,2}(\matr{A})} = \Theta(\sqrt[4]{n})$.
  
  We expect a similar behavior for the matrices of the setup in
  Section~\ref{sec:iid:matrices:1}. More precisely, we conjecture that with
  high probability the matrix~$\matr{A}$ is such that the summation
  in~\eqref{eq:key:result:1} is dominated by permutations
  $\sig_1, \sig_2 \in \setset{S}_n$ for which
  $c(\sig_1,\sig_2) = \Theta\bigl( \ln(n) \bigr)$.

\item The proofs in Appendices~\ref{sec:proof:theorem:all:one:matrix:1}
  and~\ref{sec:proof:theorem:iid:matrices:1} can be generalized and
  unified. For details, see Appendix~\ref{sec:unification:1}.

\end{itemize}

\appendices

\section{Proof of Proposition~\ref{prop:average:double:cover:1}}
\label{sec:proof:prop:average:double:cover:1}

Proposition~\ref{prop:average:double:cover:1} is obtained by
applying~\cite[Theorem~4]{7746637} to $\graphN(\matr{A})$. (Because the
following derivations rely heavily on the technique presented
in~\cite{7746637}, the reader is advised to first read that paper.)

Recall the definition of $\setset{X} \defeq \{ 0, 1 \}$ in
Section~\ref{sec:nfg:1} and the definition of
$\setset{\hat X} \defeq \setset{X} \times \setset{X} = \bigl\{ (0,0), \ (0,1),
\ (1,0), \ (1,1) \bigr\}$ in
Section~\ref{sec:double:cover:based:analysis:1}. In this appendix, we will
also use the set
\begin{align*}
  \setset{\tilde X} 
     \defeq
       \setset{X} \times \setset{X}
    &= \bigl\{ (0,0), \ (0,1), \ (1,0), \ (1,1) \bigr\}.
\end{align*}
For all these sets, their entries are considered to be ordered as shown
above. Moreover, if $\matr{M}$ is some matrix, then $\matr{M}^\tr$ is its
transpose. If $\matr{M}_1$ and $\matr{M}_2$ are two matrices of arbitrary
size, then $\matr{M}_1 \otimes \matr{M}_2$ is the Kronecker product of
$\matr{M}_1$ and $\matr{M}_2$.

We start by recalling some notations from~\cite{7746637}.
\begin{itemize}

\item Let $\setset{A} \defeq \{ a_1, a_2, \ldots, a_{|\setset{A}|} \}$ and
  $\setset{B} \defeq \{ b_1, b_2, \ldots, b_{|\setset{B}|} \}$ be some finite
  sets whose entries are considered to be ordered as shown here. We associate
  the following matrix of size $|\setset{A}| \times |\setset{B}|$ with a
  function $f: \setset{A} \times \setset{B} \to \dR$:
  \begin{align*}
    \matr{T}_f
      &\defeq
         \begin{pmatrix}
           f(a_1, b_1) & f(a_1, b_2), & \cdots & f(a_1, b_{|\setset{B}|}) \\
           f(a_2, b_1) & f(a_2, b_2), & \cdots & f(a_2, b_{|\setset{B}|}) \\
           \vdots      & \vdots      & \ddots & \vdots \\
           f(a_{|\setset{A}|}, b_1) 
             & f(a_{|\setset{A}|}, b_2) 
             & \cdots 
             & f(a_{|\setset{A}|}, b_{|\setset{B}|})
         \end{pmatrix}.
  \end{align*}
  Similarly, if $\setset{A}$, $\setset{B}$, and $\setset{C}$ are some finite
  sets, then we associate an array $\matr{T}_f$ of size
  $|\setset{A}| \times |\setset{B}| \times |\setset{C}|$ with a function
  $f: \setset{A} \times \setset{B} \times \setset{C} \to \dR$, etc.

\item The function\footnote{Note that, although $\Phi$ (which is introduced
    here) and $\tilde \Phi_M$ (which appears in
    Sections~\ref{sec:introduction:1} and~\ref{sec:nfg:1}) are both used to
    describe graph covers, they are not directly related.}
  \begin{align*}
    \Phi: \ 
      &\setset{\tilde X} \times \setset{\hat X} \to \dR
  \end{align*}
  is defined as in~\cite[Section~III]{7746637}. Namely, $\Phi$ is such that
  \begin{align*}
    \matr{T}_{\Phi}
      &\defeq
         \left(
           \begin{array}{c|cc|c}
             1 & 0          & 0           & 0 \\
           \hline
           & & & \\[-0.35cm]
             0 & 1/\sqrt{2} & 1/\sqrt{2}  & 0 \\
             0 & 1/\sqrt{2} & -1/\sqrt{2} & 0 \\
           \hline
             0 & 0          & 0           & 1
           \end{array}
         \right).
  \end{align*}
  Note that $\matr{T}_{\Phi}^\tr = \matr{T}_{\Phi}$ and
  $\matr{T}_{\Phi}^{-1} = \matr{T}_{\Phi}$.

\item As in~\cite[Section~III]{7746637}, we define the matrices
  \begin{align*}
    \matr{\tilde{E}}_{\mathrm{nocross}}
      &\defeq
         \left(
           \begin{array}{c|cc|c}
             1 & 0 & 0 & 0 \\
           \hline
             0 & 1 & 0 & 0 \\
             0 & 0 & 1 & 0 \\
           \hline
             0 & 0 & 0 & 1
           \end{array}
         \right), \\
    \matr{\tilde{E}}_{\mathrm{cross}}
      &\defeq
         \left(
           \begin{array}{c|cc|c}
             1 & 0 & 0 & 0 \\
           \hline
             0 & 0 & 1 & 0 \\
             0 & 1 & 0 & 0 \\
           \hline
             0 & 0 & 0 & 1
           \end{array}
         \right),
  \end{align*}
  whose rows and columns are labeled by the elements of $\setset{\tilde
    X}$. Based on these matrices, we define the matrices
  \begin{alignat*}{2}
    \matr{\hat{E}}_{\mathrm{nocross}}
      &\defeq
         \matr{T}_{\Phi}^\tr
           \cdot 
           \matr{\tilde{E}}_{\mathrm{nocross}}
           \cdot 
           \matr{T}_{\Phi}
     &&= \left(
           \begin{array}{c|cc|c}
             1 & 0 & 0 & 0 \\
           \hline
             0 & 1 & 0 & 0 \\
             0 & 0 & 1 & 0 \\
           \hline
             0 & 0 & 0 & 1
           \end{array}
         \right), \\
    \matr{\hat{E}}_{\mathrm{cross}}
      &\defeq
    \matr{T}_{\Phi}^\tr
      \cdot 
      \matr{\tilde{E}}_{\mathrm{cross}}
      \cdot 
      \matr{T}_{\Phi}
    &&= \left(
           \begin{array}{c|cc|c}
             1 & 0 &  0 & 0 \\
           \hline
             0 & 1 &  0 & 0 \\
             0 & 0 & -1 & 0 \\
           \hline
             0 & 0 &  0 & 1
           \end{array}
         \right),
  \end{alignat*}
  and
  \begin{align*}
    \matr{T}_{\hat{E}_e}
      &\defeq
         \frac{1}{2}
           \cdot
           \matr{\hat{E}}_{\mathrm{nocross}}
         +
         \frac{1}{2}
           \cdot
           \matr{\hat{E}}_{\mathrm{cross}}
       = \left(
             \begin{array}{c|cc|c}
               1 & 0 &  0 & 0 \\
             \hline
               0 & 1 &  0 & 0 \\
               0 & 0 &  0 & 0 \\
             \hline
               0 & 0 &  0 & 1
             \end{array}
           \right).
  \end{align*}
  whose rows and columns are labeled by the elements of~$\setset{\hat X}$.

\end{itemize}

In order to apply the technique from~\cite{7746637}, we need, in a first step,
to compute the functions $\tfL{i}$ and $\hfL{i}$ derived from $\fL{i}$,
$i \in [n]$, and the functions $\tfR{j}$ and $\hfR{j}$ derived from $\fR{j}$,
$j \in [n]$. Here we will only find $\tfL{i}$ and $\hfL{i}$, $i \in [n]$,
because $\tfR{j}$ and $\hfR{j}$, $j \in [n]$, are obtained in an analogous
manner.

Fix some positive integer $n$ and some $i \in [n]$. According
to~\cite{7746637}, $\tfL{i}$ and $\hfL{i}$ are defined as, respectively,
\begin{alignat*}{3}
  \tfL{i}: \ 
    & \setset{\tilde{X}}^{n}
   && \to
   && \ \dR \\
    & (\tilde{x}_{i,1}, \ldots, \tilde{x}_{i,n})
   && \mapsto
   && \ 
      f(x_{i,1}, \ldots, x_{i,n})
      \cdot
      f(x'_{i,1}, \ldots, x'_{i,n}), \\[0.25cm]
  \hfL{i}: \ 
    & \setset{\hat{X}}^{n}
   && \to
   && \ \dR \\
    & (\hat{x}_{i,1}, \ldots, \hat{x}_{i,n})
   && \mapsto
   && \sum_{\tilde{x}_1, \ldots, \tilde{x}_{n} \in \setset{\tilde X}}
        \tfL{i}(\tilde{x}_{i,1}, \ldots, \tilde{x}_{i,n}) \\
    &
   &&
   && \hspace{1.75cm}
        \cdot \ 
        \prod_{j=1}^{n}
          \Phi(\tilde{x}_{i,j}, \hat{x}_{i,j}),
\end{alignat*}
where we have used
$\tilde{x}_{i,j} \defeq (x_{i,j}, x'_{i,j}) \in \setset{\tilde X}$,
$j \in [n]$, in the definition of $\tfL{i}$.

In the following, we want to find the explicit expression for
$\hfL{i}$. However, before tackling the case of general $n$, we will first
consider the cases $n = 2$ and $n = 3$.

\subsection{Function $\hfL{i}$: case $n = 2$}
\label{sec:proof:proposition:1:case:n:2}

Let $n = 2$ and fix some $i \in [n]$. The matrix associated with the function
$\fL{i}$ turns out to be
\begin{align*}
  \matr{T}_{\fL{i}}
    &\defeq
       \begin{pmatrix}
         0            & \sqrt{a_{i,2}} \\
         \sqrt{a_{i,1}} & 0
       \end{pmatrix}.
\end{align*}
Reusing some of the calculations in~\cite[Section~IV]{7746637}, the matrices
associated with $\tfL{i}$ and $\hfL{i}$ are, respectively,
\begin{align*}
  \matr{T}_{\tfL{i}}
    &\! \defeq \!
       \matr{T}_{\fL{i}} 
       \otimes 
       \matr{T}_{\fL{i}}
     \! = \!
       {\scriptsize
       \left(
         \begin{array}{c|cc|c}
         0       & 0                     & 0                     & a_{i,2} 
         \\[0.08cm]
         \hline
         0       & 0                     & \sqrt{a_{i,2} a_{i,1}} & 0 
         \\
         0       & \sqrt{a_{i,1} a_{i,2}} & 0 & 0
         \\[0.08cm]
         \hline
         a_{i,1} & 0                      & 0 & 0
       \end{array}
       \right)
       } \! , \\
  \matr{T}_{\hfL{i}}
    &\! \defeq \!
       \matr{T}_{\Phi}
         \cdot 
         \matr{T}_{\tfL{i}} \!\!\!
         \cdot 
         \matr{T}_{\Phi}
     \! = \! 
     {\scriptsize
       \left(
         \begin{array}{c|c@{\ }c|c}
           0
             & 0
             & 0 
             & a_{i,1} \\[0.08cm]
         \hline
         & & & \\[-0.25cm]
           0
             & \sqrt{a_{i,1} a_{i,2}}
             & 0 
             & 0 \\
           0 
             & 0 
             & - \sqrt{a_{i,1} a_{i,2}}
             & 0 \\[0.08cm]
         \hline
         & & & \\[-0.25cm]
           a_{i,2}
             & 0
             & 0 
             & 0
         \end{array}
       \right)
       } \! ,
\end{align*}
whose rows and columns are labeled by the elements of $\setset{\hat X}$.

Because all the entries of $\matr{T}_{\hat{E}_e}$ in row $(1,0)$ and all
entries of $\matr{T}_{\hat{E}_e}$ in column $(1,0)$ are equal to zero, it
turns out that $Z\bigl( \hat \graphN(\matr{A}) \bigr)$ is unchanged if the
function $\hfL{i}$ is replaced by the function $\hfL{i}$ such that
\begin{align*}
  \matr{T}_{\hfL{i}}
    &= \left(
         \begin{array}{c|cc|c}
           0
             & 0
             & 0 
             & a_{i,1} \\[0.08cm]
         \hline
         & & & \\[-0.25cm]
           0
             & \sqrt{a_{i,1} a_{i,2}}
             & 0 
             & 0 \\
           0 
             & 0 
             & 0
             & 0 \\[0.08cm]
         \hline
         & & & \\[-0.25cm]
           a_{i,2}
             & 0
             & 0 
             & 0
         \end{array}
       \right).
\end{align*}
The reader can verify that this matches the function $\hfL{i}$ stated in
Proposition~\ref{prop:average:double:cover:1} for the case $n = 2$.

\subsection{Function $\hfL{i}$: case $n = 3$}
\label{sec:proof:proposition:1:case:n:3}

Consider now the case $n = 3$ and fix some $i \in [n]$. (Note that in the
following we do not show the arrays associated with $\fL{i}$ and $\tfL{i}$ and
directly show the array associated with $\hfL{i}$.)

Reusing some of the calculations in~\cite[Section~IV]{7746637}, the array
$\matr{T}_{\hfL{i}}$ of size $4 \times 4 \times 4$ associated with $\hfL{i}$
is\footnote{Note that there were some typos in the paragraph before Lemma~6
  in~\cite{7746637}. Namely, the second occurrence of
  $\det(\matr{T}_{f|a_1=0})$ should be $\det(\matr{T}_{f|a_2=0})$ and the
  third occurrence of $\det(\matr{T}_{f|a_1=0})$ should be
  $\det(\matr{T}_{f|a_2=1})$. Moreover, the last term in the expression for
  $\hat{f}(\hzero,\hzero,\hzero)$ should be $t_{001} t_{110}$ instead of
  $t_{000} t_{110}$.}

{\scriptsize

\begin{center}
  $\displaystyle
  \left(
    \begin{array}{C{1.70cm}|C{1.70cm}C{1.70cm}|C{1.70cm}}
      0
        & 0
        & 0
        & a_{i,2} \\[0.08cm]
    \hline
     & & & \\[-0.25cm]
      0
        & \sqrt{a_{i,1} a_{i,2}}
        & 0 
        & 0 \\
      0 
        & 0 
        & - \sqrt{a_{i,1} a_{i,2}}
        & 0 \\[0.08cm]
    \hline
    & & & \\[-0.25cm]
      a_{i,1}
        & 0
        & 0 
        & 0
    \end{array}
  \right)$, \\[0.35cm]
  $\displaystyle
  \left(
    \begin{array}{C{1.70cm}|C{1.70cm}C{1.70cm}|C{1.70cm}}
      0
        & \sqrt{a_{i,2} a_{i,3}}
        & 0 
        & 0 \\[0.08cm]
    \hline
    & & & \\[-0.25cm]
       \sqrt{a_{i,1} a_{i,3}}
        & 0
        & 0 
        & 0 \\
      0 
        & 0 
        & 0
        & 0 \\[0.08cm]
    \hline
    & & & \\[-0.25cm]
      0
        & 0
        & 0 
        & 0
    \end{array}
  \right)$, \\[0.35cm]
  $\displaystyle
  \left(
    \begin{array}{C{1.70cm}|C{1.70cm}C{1.70cm}|C{1.70cm}}
      0 
        & 0 
        & - \sqrt{a_{i,2} a_{i,3}}
        & 0 \\[0.08cm]
    \hline
    & & & \\[-0.25cm]
      0 
        & 0 
        & 0
        & 0 \\
      - \sqrt{a_{i,1} a_{i,3}}
        & 0
        & 0 
        & 0 \\[0.08cm]
    \hline
    & & & \\[-0.25cm]
      0 
        & 0 
        & 0
        & 0
    \end{array}
  \right)$, \\[0.35cm]
  $\displaystyle
  \left(
    \begin{array}{C{1.70cm}|C{1.70cm}C{1.70cm}|C{1.70cm}}
      a_{i,3}
        & 0
        & 0 
        & 0 \\[0.08cm]
    \hline
    & & & \\[-0.25cm]
      0
        & 0
        & 0 
        & 0 \\
      0 
        & 0 
        & 0
        & 0 \\[0.08cm]
    \hline
    & & & \\[-0.25cm]
      0
        & 0
        & 0 
        & 0
    \end{array}
  \right)$, \\[0.35cm]
\end{center}

}

\noindent
where all three dimensions are labeled by the elements of $\setset{\hat X}$.

Because all the entries of $\matr{T}_{\hat{E}_e}$ in row $(1,0)$ and all
entries of $\matr{T}_{\hat{E}_e}$ in column $(1,0)$ are equal to zero, it
turns out that $Z\bigl( \hat \graphN(\matr{A}) \bigr)$ is unchanged if the
function $\hfL{i}$ is replaced by the function $\hfL{i}$ such that the array
$\matr{T}_{\hfL{i}}$ of size $4 \times 4 \times 4$ associated with $\hfL{i}$
is

{\scriptsize

\begin{center}
  $\displaystyle
  \left(
    \begin{array}{C{1.70cm}|C{1.70cm}C{1.70cm}|C{1.70cm}}
      0
        & 0
        & 0
        & a_{i,2} \\[0.08cm]
    \hline
     & & & \\[-0.25cm]
      0
        & \sqrt{a_{i,1} a_{i,2}}
        & 0 
        & 0 \\
      0 
        & 0 
        & 0
        & 0 \\[0.08cm]
    \hline
    & & & \\[-0.25cm]
      a_{i,1}
        & 0
        & 0 
        & 0
    \end{array}
  \right)$, \\[0.35cm]
  $\displaystyle
  \left(
    \begin{array}{C{1.70cm}|C{1.70cm}C{1.70cm}|C{1.70cm}}
      0
        & \sqrt{a_{i,2} a_{i,3}}
        & 0 
        & 0 \\[0.08cm]
    \hline
    & & & \\[-0.25cm]
       \sqrt{a_{i,1} a_{i,3}}
        & 0
        & 0 
        & 0 \\
      0 
        & 0 
        & 0
        & 0 \\[0.08cm]
    \hline
    & & & \\[-0.25cm]
      0
        & 0
        & 0 
        & 0
    \end{array}
  \right)$, \\[0.35cm]
  $\displaystyle
  \left(
    \begin{array}{C{1.70cm}|C{1.70cm}C{1.70cm}|C{1.70cm}}
      0 
        & 0 
        & 0
        & 0 \\[0.08cm]
    \hline
    & & & \\[-0.25cm]
      0 
        & 0 
        & 0
        & 0 \\
      0
        & 0
        & 0 
        & 0 \\[0.08cm]
    \hline
    & & & \\[-0.25cm]
      0 
        & 0 
        & 0
        & 0
    \end{array}
  \right)$, \\[0.35cm]
  $\displaystyle
  \left(
    \begin{array}{C{1.70cm}|C{1.70cm}C{1.70cm}|C{1.70cm}}
      a_{i,3}
        & 0
        & 0 
        & 0 \\[0.08cm]
    \hline
    & & & \\[-0.25cm]
      0
        & 0
        & 0 
        & 0 \\
      0 
        & 0 
        & 0
        & 0 \\[0.08cm]
    \hline
    & & & \\[-0.25cm]
      0
        & 0
        & 0 
        & 0
    \end{array}
  \right)$, \\[0.35cm]
\end{center}

}

\noindent
The reader can verify that this matches the function $\hfL{i}$ stated in
Proposition~\ref{prop:average:double:cover:1} for the case $n = 3$.

\subsection{Function $\hfL{i}$: general $n$}
\label{sec:proof:proposition:1:case:n:general}

We now consider an arbitrary $n \in \dN$ and fix some $i \in [n]$.

The following characterization of the function $\tfL{i}$ follows immediately
from the properties of the function $\fL{i}$. (Recall the definition of
$\fL{i}$ in~Section~\ref{sec:nfg:1}.)
\begin{itemize}

\item If there exists a $j \in [n]$ s.t.\ $\tilde{x}_{i,j} = (1,1)$ and
  $\tilde{x}_{i,j'} = (0,0)$ for all $j' \in [n] \setminus \{ j \}$, then
  \begin{align}
    \tfL{i}(\tilde{x}_{i,1}, \ldots, \tilde{x}_{i,n})
      &= \sqrt{a_{i,j}} \cdot \sqrt{a_{i,j}} 
       = a_{i,j}.
           \label{eq:property:tilde:f:L:1}
  \end{align}

\item If there exist $j, j' \in [n]$, $j \neq j'$, s.t.\ 
  $\tilde{x}_{i,j} = (0,1)$, $\tilde{x}_{i,j'} = (1,0)$ and
  $\tilde{x}_{i,j''} = (0,0)$ for all $j'' \in [n] \setminus \{ j, j' \}$,
  then
  \begin{align}
    \tfL{i}(\tilde{x}_{i,1}, \ldots, \tilde{x}_{i,n})
      &= \sqrt{a_{i,j} \cdot a_{i,j'}}.
           \label{eq:property:tilde:f:L:2}
  \end{align}

\item Otherwise,
  \begin{align}
    \tfL{i}(\tilde{x}_{i,1}, \ldots, \tilde{x}_{i,n})
      &= 0.
           \label{eq:property:tilde:f:L:3}
  \end{align}

\end{itemize}
Based on this, the function $\hfL{i}$ can be characterized as follows.
\begin{itemize}

\item If there exists a $j \in [n]$ s.t.\ $\hat{x}_{i,j} = (1,1)$ and
  $\hat{x}_{i,j'} = (0,0)$ for all $j' \in [n] \setminus \{ j \}$, then
  \begin{align*}
    \hfL{i}(\hat{x}_{i,1}, \ldots, \hat{x}_{i,n})
      &= a_{i,j}.
  \end{align*}
  This follows from~\eqref{eq:property:tilde:f:L:1} and
  \begin{align*}
    \Phi(\tilde{x}_{i,j},\hat{x}_{i,j}) 
      &= \begin{cases}
           1 & \text{if $\tilde{x}_{i,j} = (1,1)$} \\
           0 & \text{otherwise}
         \end{cases}, \\
    \Phi(\tilde{x}_{i,j'},\hat{x}_{i,j'}) 
      &= \begin{cases}
           1 & \text{if $\tilde{x}_{i,j'} = (0,0)$} \\
           0 & \text{otherwise}
         \end{cases}.
  \end{align*}

\item If there exist $j, j' \in [n]$, $j \neq j'$, s.t.\
  $\hat{x}_{i,j} = (0,1)$, $\hat{x}_{i,j'} = (0,1)$ and
  $\hat{x}_{i,j''} = (0,0)$ for all $j'' \in [n] \setminus \{ j, j' \}$, then
  \begin{align*}
    \hfL{i}(\hat{x}_{i,1}, \ldots, \hat{x}_{i,n})
      &= \sqrt{a_{i,j} \cdot a_{i,j'}}.
  \end{align*}
  This follows from~\eqref{eq:property:tilde:f:L:2} and
  \begin{align*}
    \Phi(\tilde{x}_{i,j},\hat{x}_{i,j}) 
    &= \begin{cases}
         \frac{1}{\sqrt{2}} 
           & \text{if $\tilde{x}_{i,j} \in \bigl\{ (0,1), \ (1,0) \bigr\}$} \\
         0 & \text{otherwise}
       \end{cases}, \\
    \Phi(\tilde{x}_{i,j'},\hat{x}_{i,j'}) 
    &= \begin{cases}
         \frac{1}{\sqrt{2}} 
           & \text{if $\tilde{x}_{i,j'} \in \bigl\{ (0,1), \ (1,0) \bigr\}$} \\
         0 & \text{otherwise}
       \end{cases}, \\
    \Phi(\tilde{x}_{i,j''},\hat{x}_{i,j''}) 
      &= \begin{cases}
           1 & \text{if $\tilde{x}_{i,j''} = (0,0)$} \\
           0 & \text{otherwise}
         \end{cases}.
  \end{align*}

\item If there exist $j, j' \in [n]$, $j \neq j'$, s.t.\
  $\hat{x}_{i,j} = (1,0)$, $\hat{x}_{i,j'} = (1,0)$ and
  $\hat{x}_{i,j''} = (0,0)$ for all $j'' \in [n] \setminus \{ j, j' \}$, then
  \begin{align*}
    \hfL{i}(\hat{x}_{i,1}, \ldots, \hat{x}_{i,n})
      &= - \sqrt{a_{i,j} \cdot a_{i,j'}}.
  \end{align*}
  This follows from~\eqref{eq:property:tilde:f:L:2} and
  \begin{align*}
    \Phi(\tilde{x}_{i,j},\hat{x}_{i,j}) 
    &= \begin{cases}
         \frac{1}{\sqrt{2}} 
           & \text{if $\tilde{x}_{i,j} = (0,1)$} \\
         -\frac{1}{\sqrt{2}} 
           & \text{if $\tilde{x}_{i,j} = (1,0)$} \\
         0 & \text{otherwise}
       \end{cases}, \\
    \Phi(\tilde{x}_{i,j'},\hat{x}_{i,j'}) 
    &= \begin{cases}
         \frac{1}{\sqrt{2}} 
           & \text{if $\tilde{x}_{i,j'} = (0,1)$} \\
         -\frac{1}{\sqrt{2}} 
           & \text{if $\tilde{x}_{i,j'} = (1,0)$} \\
         0 & \text{otherwise}
       \end{cases}, \\
    \Phi(\tilde{x}_{i,j''},\hat{x}_{i,j''}) 
      &= \begin{cases}
           1 & \text{if $\tilde{x}_{i,j''} = (0,0)$} \\
           0 & \text{otherwise}
         \end{cases}.
  \end{align*}

\item Otherwise,
  \begin{align*}
    \hfL{i}(\hat{x}_{i,1}, \ldots, \hat{x}_{i,n})
      &= 0.
  \end{align*}
  This follows from~\eqref{eq:property:tilde:f:L:3} and/or from the exact
  cancellation of terms.

\end{itemize}

Finally, we can make a similar observation as in
Sections~\ref{sec:proof:proposition:1:case:n:2}
and~\ref{sec:proof:proposition:1:case:n:3}. Namely, because all the entries of
$\matr{T}_{\hat{E}_e}$ in row $(1,0)$ and all entries of
$\matr{T}_{\hat{E}_e}$ in column $(1,0)$ are equal to zero, it turns out that
$Z\bigl( \hat \graphN(\matr{A}) \bigr)$ is unchanged if the function $\hfL{i}$
is replaced by the function $\hfL{i}$ where
$\hfL{i}(\hat{x}_{i,1}, \ldots, \hat{x}_{i,n})$ is set to $0$ if there exists
a $j \in [n]$ such that $\hat{x}_{i,j} = (1,0)$. We then obtain
\begin{itemize}

\item If there exists a $j \in [n]$ s.t.\ $\hat{x}_{i,j} = (1,1)$ and
  $\hat{x}_{i,j'} = (0,0)$ for all $j' \in [n] \setminus \{ j \}$, then
  \begin{align*}
    \hfL{i}(\hat{x}_{i,1}, \ldots, \hat{x}_{i,n})
      &= a_{i,j}.
  \end{align*}

\item If there exist $j, j' \in [n]$, $j \neq j'$, s.t.\
  $\hat{x}_{i,j} = (0,1)$, $\hat{x}_{i,j'} = (0,1)$ and
  $\hat{x}_{i,j''} = (0,0)$ for all $j'' \in [n] \setminus \{ j, j' \}$, then
  \begin{align*}
    \hfL{i}(\hat{x}_{i,1}, \ldots, \hat{x}_{i,n})
      &= \sqrt{a_{i,j} \cdot a_{i,j'}}.
  \end{align*}

\item Otherwise,
  \begin{align*}
    \hfL{i}(\hat{x}_{i,1}, \ldots, \hat{x}_{i,n})
      &= 0.
  \end{align*}

\end{itemize}
The reader can verify that this matches the function $\hfL{i}$ stated in
Proposition~\ref{prop:average:double:cover:1} for arbitrary $n \in \dN$.

As mentioned earlier, here we only showed how to find $\tfL{i}$ and $\hfL{i}$,
$i \in [n]$, because $\tfR{j}$ and $\hfR{j}$, $j \in [n]$, are obtained in an
analogous manner.

\section{Proof of Proposition~\ref{prop:average:double:cover:2}}
\label{sec:proof:prop:average:double:cover:2}

In this appendix, we want to prove that
\begin{align*}
  \frac{\perm_{\Bethe,2}(\matr{A})}{\perm(\matr{A})}
   & = \sqrt
       {
         \sum_{\sig_1, \sig_2 \in \Sn}
           p(\sig_1)
           \cdot
           p(\sig_2)
           \cdot
           2^{-c(\sig_1,\sig_2)}
       },
\end{align*}
which can be rewritten as
\begin{align*}
  & \hspace{-0.50cm}
  \perm_{\Bethe, 2}(\matr{A})^2 \\
    &= \!\!\!\!\! 
       \sum_{\sig_1, \sig_2 \in \Sn}
       \!\!\!\!\!
         2^{-c(\sig_1,\sig_2)}
         \cdot
         \left(
           \prod_{i \in [n]}
             a_{i,\sig_1(i)}
         \right)
         \cdot
         \left(
           \prod_{i \in [n]}
             a_{i,\sig_2(i)}
         \right).
\end{align*}
However, because of Proposition~\ref{prop:average:double:cover:1}, this is
equivalent to proving
\begin{align}
  Z(\hat{\graphN}( \matr{A}) \bigr)
    &= \!\!\!\!\! 
       \sum_{\sig_1, \sig_2 \in \Sn}
       \!\!\!\!\! 
         2^{-c(\sig_1,\sig_2)}
         \! \cdot \!
         \left(
           \prod_{i \in [n]}
             a_{i,\sig_1(i)}
         \right)
         \! \cdot \!
         \left(
           \prod_{i \in [n]}
             a_{i,\sig_2(i)}
         \right) \! .
           \label{eq:proposition:reformulation:1}
\end{align}
In order to prove~\eqref{eq:proposition:reformulation:1}, we need to better
understand the valid configurations of $\graph{\hat N}(\matr{A})$ and their
global function value.

Consider the NFG $\graph{\hat N}(\matr{A})$ as specified in
Proposition~\ref{prop:average:double:cover:1}. Recall that its global function
is called $\hat{g}(\hvx)$, where
$\hvx \defeq (\hat{x}_{1,1}, \hat{x}_{1,2}, \ldots, \hat{x}_{n,n}) \in
\setset{\hat X}^{(n^2)}$ is an arbitrary configuration of
$\graph{\hat N}(\matr{A})$. Let
$\setset{C}\bigl( \graph{\hat N}(\matr{A}) \bigr)$ be the set of all valid
configurations of $\graph{\hat N}(\matr{A})$, i.e., the set of all
$\hvx \in \setset{\hat X}^{(n^2)}$ such that
$\hat{g}(\hvx) \neq 0$.\footnote{In the following, we assume that all entries
  of $\matr{A}$ are positive. The case where some entries of $\matr{A}$ are
  zero can be handled by suitable adaptations, or by considering a sequence of
  matrices with positive entries where some entries converge to $0$ in the
  limit and using continuity of the relevant expressions.}

Consider Fig.~\ref{fig:ffg:permanent:valid:configuration:1}, which is a
reproduction of~Fig.~\ref{fig:ffg:permanent:1}~(right). Let
$\hvx \in \setset{C}\bigl( \graph{\hat N}(\matr{A}) \bigr)$ be a valid
configuration. From Proposition~\ref{prop:average:double:cover:1} it follows
that every vertex of $\graph{\hat N}(\matr{A})$ is either
\begin{itemize}

\item the endpoint of exactly a $(1,1)$-edge, where a $(1,1)$-edge is defined
  to be an edge such that $\hat{x}_{i,j} = (1,1)$, or

\item a vertex of exactly one $(0,1)$-cycle, where a $(0,1)$-cycle is defined
  to be a simple cycle such that $\hat{x}_{i,j} = (0,1)$ for all its edges.

\end{itemize}

\begin{example}
  \label{example:valid:configuration:1}

  The NFG in Fig.~\ref{fig:ffg:permanent:valid:configuration:1} shows a
  possible valid configuration of $\graph{\hat N}(\matr{A})$ for the case
  $n = 5$. Here, an edge is colored in blue if $\hat{x}_{i,j} = (1,1)$, it is
  colored in red if $\hat{x}_{i,j} = (0,1)$, and it is colored in black if
  $\hat{x}_{i,j} = (0,0)$. (Note that $\hat{x}_{i,j} = (1,0)$ cannot occur in
  a valid configuration.)
\end{example}

\begin{figure}
  \begin{center}
    \includegraphics[width=0.45\linewidth]%
                    {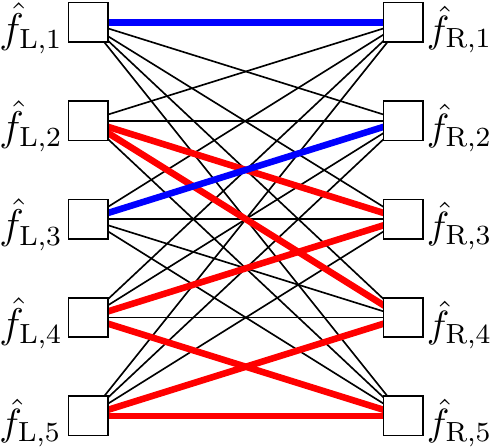}
  \end{center}
  \vspace{0.25cm}
  \caption{NFG used in
    Appendix~\ref{sec:proof:prop:average:double:cover:2}. The two
    $(1,1)$-edges are colored in blue and the edges of the $(0,1)$-cycle are
    colored in red.}
  \label{fig:ffg:permanent:valid:configuration:1}
\end{figure}

\begin{figure}
  \begin{center}
    \begin{subfigure}{0.45\linewidth}
      \includegraphics[width=0.85\linewidth]%
                      {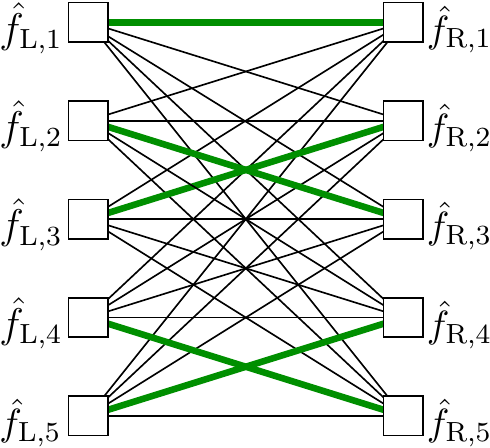}
    \end{subfigure}
    \begin{subfigure}{0.45\linewidth}
      \includegraphics[width=0.85\linewidth]%
                      {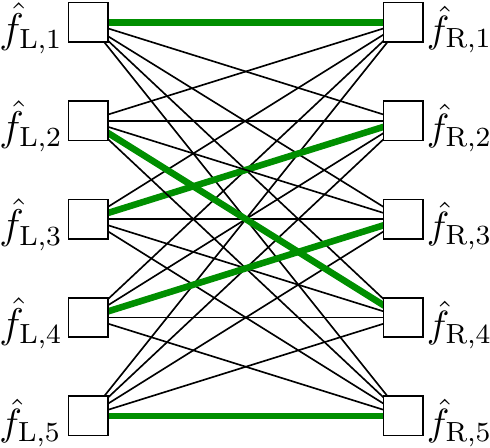}
    \end{subfigure}
  \end{center}
  \vspace{0.25cm}
  \caption{NFGs used in
    Appendix~\ref{sec:proof:prop:average:double:cover:2}. Visualization of
    permutations $\sig \in \Sn$. For $i, j \in [n]$, the edge connecting
    $\hfL{i}$ to $\hfR{j}$ is colored in green if and only if $\sig(i) =
    j$. (Here, $n = 5$.)}
  \label{fig:ffg:permanent:valid:configuration:2}
  \vspace{-0.5cm}
\end{figure}

\begin{definition}
  Consider the mapping
  \begin{alignat*}{3}
    h: \ 
      & \Sn \times \Sn
     && \to
     && \ \setset{C}\bigl( \graph{\hat N}(\matr{A}) \bigr) \\
      & (\sig_1, \sig_2)
     && \mapsto
     && \ \hvx,
  \end{alignat*}
  where $\hat{x}_{i,j}$, $i,j \in [n]$, is defined as follows:
  \begin{align*}
    \hat{x}_{i,j}
      &\defeq
         \begin{cases}
           (1,1) & \text{if $j = \sig_1(i) = \sig_2(i)$} \\
           (0,1) & \text{if $j = \sig_1(i) \neq \sig_2(i)$} \\
           (0,1) & \text{if $j = \sig_2(i) \neq \sig_1(i)$} \\
           (0,0) & \text{otherwise}
         \end{cases}.
  \end{align*}
  (It can be readily verified that, indeed,
  $\hvx \in \setset{C}\bigl( \graph{\hat N}(\matr{A}) \bigr)$.)
\end{definition}

\begin{example}
  \label{example:permutation:selection:1}

  One can verify that if $\sig_1, \sig_2 \in \Sn$ are such that
  \begin{align*}
    &
    \sig_1(1) = 1, \ 
    \sig_1(2) = 3, \ 
    \sig_1(3) = 2, \ 
    \sig_1(4) = 5, \ 
    \sig_1(5) = 4, \\
    &
    \sig_2(1) = 1, \ 
    \sig_2(2) = 4, \ 
    \sig_2(3) = 2, \ 
    \sig_2(4) = 3, \ 
    \sig_2(5) = 5,
  \end{align*}
  then $h(\sig_1, \sig_2)$ equals the valid configuration that is
  highlighted in the NFG in
  Fig.~\ref{fig:ffg:permanent:valid:configuration:1}. Note that $\sig_1$ and
  $\sig_2$ are visualized in
  Fig.~\ref{fig:ffg:permanent:valid:configuration:2}~(left) and
  Fig.~\ref{fig:ffg:permanent:valid:configuration:2}~(right),
  respectively.
\end{example}

\begin{example}
  \label{example:permutation:selection:2}

  Similarly, one can verify that if $\sig_1, \sig_2 \in \Sn$ are such that
  \begin{align*}
    &
    \sig_1(1) = 1, \ 
    \sig_1(2) = 4, \ 
    \sig_1(3) = 2, \ 
    \sig_1(4) = 3, \ 
    \sig_1(5) = 5, \\
    &
    \sig_2(1) = 1, \ 
    \sig_2(2) = 3, \ 
    \sig_2(3) = 2, \ 
    \sig_2(4) = 5, \ 
    \sig_2(5) = 4,
  \end{align*}
  then $h(\sig_1, \sig_2)$ equals the valid configuration that is
  highlighted in the NFG in
  Fig.~\ref{fig:ffg:permanent:valid:configuration:1}. Note that $\sig_1$ and
  $\sig_2$ are visualized in
  Fig.~\ref{fig:ffg:permanent:valid:configuration:2}~(right) and
  Fig.~\ref{fig:ffg:permanent:valid:configuration:2}~(left), respectively.
\end{example}

As can be seen from Examples~\ref{example:permutation:selection:1}
and~\ref{example:permutation:selection:2}, a $(0,1)$-cycle in the valid
configuration $h(\sig_1,\sig_2)$ arises from alternatingly picking an edge
selected by $\sig_1$ and an edge selected by $\sig_2$.

Toward proving the upcoming Lemmas~\ref{lemma:number:of:cycles:1}
and~\ref{lemma:number:of:pre:images:1}, we consider the following examples.

\begin{example}
  \label{example:permutation:selection:3}

  Consider again the selection of $\sig_1, \sig_2 \in \Sn$ in
  Example~\ref{example:permutation:selection:1}. Let
  $\sig \defeq \sig_1 \circ \sig_2^{-1}$, i.e., the permutation obtained by
  first applying the inverse of $\sig_2$ and then $\sig_1$. Then
  \begin{align*}
    &
    \sig(1) = 1, \ 
    \sig(2) = 2, \ 
    \sig(3) = 5, \ 
    \sig(4) = 3, \ 
    \sig(5) = 4,
  \end{align*}
  which in cycle notation is
  \begin{align*}
    \sig
      &= (1)(2)(354).
  \end{align*}
  Note that
  \begin{itemize}

  \item the cycle $(1)$ of $\sig$ of length~$1$ corresponds to the
    blue-colored edge ending in $\hfR{1}$ in
    Fig.~\ref{fig:ffg:permanent:valid:configuration:1},

  \item the cycle $(2)$ of $\sig$ of length~$1$ corresponds to the
    blue-colored edge ending in $\hfR{2}$ in
    Fig.~\ref{fig:ffg:permanent:valid:configuration:1},

  \item the cycle $(354)$ of $\sig$ of length~$3$ corresponds to the
    blue-colored $(0,1)$-cycle in
    Fig.~\ref{fig:ffg:permanent:valid:configuration:1} of length
    $2 \cdot 3 = 6$ going through the right-hand side vertices $\hfR{3}$,
    $\hfR{5}$, $\hfR{4}$.

  \end{itemize}
\end{example}

\begin{example}
  \label{example:permutation:selection:4}

  Consider again the selection of $\sig_1, \sig_2 \in \Sn$ in
  Example~\ref{example:permutation:selection:2}. Let
  $\sig \defeq \sig_1 \circ \sig_2^{-1}$. Then
  \begin{align*}
    &
    \sig(1) = 1, \ 
    \sig(2) = 2, \ 
    \sig(3) = 4, \ 
    \sig(4) = 5, \ 
    \sig(5) = 3,
  \end{align*}
  which in cycle notation is
  \begin{align*}
    \sig
      &= (1)(2)(345).
  \end{align*}
  Note that
  \begin{itemize}

  \item the cycle $(1)$ of $\sig$ of length~$1$ corresponds to the
    blue-colored edge ending in $\hfR{1}$ in
    Fig.~\ref{fig:ffg:permanent:valid:configuration:1},

  \item the cycle $(2)$ of $\sig$ of length~$1$ corresponds to the
    blue-colored edge ending in $\hfR{2}$ in
    Fig.~\ref{fig:ffg:permanent:valid:configuration:1},

  \item the cycle $(345)$ of $\sig$ of length~$3$ corresponds to the
    blue-colored $(0,1)$-cycle in
    Fig.~\ref{fig:ffg:permanent:valid:configuration:1} of length
    $2 \cdot 3 = 6$ going through the right-hand side vertices $\hfR{3}$,
    $\hfR{4}$, $\hfR{5}$.

  \end{itemize}
\end{example}

For $\sig_1, \sig_2 \in \Sn$, recall from
Proposition~\ref{sec:proof:prop:average:double:cover:2} that
$c(\sig_1,\sig_2)$ is defined to be the number of cycles of length larger than
one in the cycle notation expression of the permutation
$\sig_1 \circ \sig_2^{-1}$.

\begin{lemma}
  \label{lemma:number:of:cycles:1}

  For any $\sig_1, \sig_2 \in \Sn$, the number of $(0,1)$-cycles in the valid
  configuration $h(\sig_1, \sig_2)$ equals $c(\sig_1, \sig_2)$.
\end{lemma}

\begin{proof}
  This follows from generalizing
  Examples~\ref{example:permutation:selection:3}
  and~\ref{example:permutation:selection:4}. In particular, note that a cycle
  of length $L$, $L \geq 2$, of $\sig_1 \circ \sig_2^{-1}$ corresponds to a
  $(0,1)$-cycle of length $2 L$ in the valid configuration $h(\sig_1, \sig_2)$
\end{proof}

\begin{lemma}
  \label{lemma:number:of:pre:images:1}

  Fix an arbitrary
  $\hvx \in \setset{C}\bigl( \graph{\hat N}(\matr{A}) \bigr)$. The number of
  pre-images of $\hvx$ under the mapping $h$ equals
  $2^{c(\sig_1,\sig_2)}$, i.e.,
  \begin{align*}
    \Bigl|
      \bigl\{ 
        (\sig_1, \sig_2) \in \Sn \times \Sn
      \bigm|
        h(\sig_1, \sig_2) = \hvx
      \bigr\}
    \Bigr|
      &= 2^{c(\sig_1,\sig_2)}.
  \end{align*}
\end{lemma}

\begin{proof}
  This follows from generalizing
  Examples~\ref{example:permutation:selection:1}--\ref{example:permutation:selection:4}. In
  particular, note that for every $(0,1)$-cycle in $\hvx$ there are two walk
  directions for going around the $(0,1)$-cycle, and with that two different
  ways of defining the relevant function values of $\sig_1$ and
  $\sig_2$. Because the walk direction can be chosen independently for every
  $(0,1)$-cycle in $\hvx$, there are, ultimately, $2^{c(\sig_1,\sig_2)}$ ways
  of selecting $\sig_1, \sig_2 \in \Sn$ such that $h(\sig_1,\sig_2) = \hvx$.
\end{proof}

\begin{lemma}
  \label{lemma:global:function:value:1}

  For any $\sig_1, \sig_2 \in \Sn$, it holds that
  \begin{align*}
    \hat{g}\bigl( h(\sig_1,\sig_2) \bigr)
      &= \left(
           \prod_{i \in [n]}
             a_{i,\sig_1(i)}
         \right)
         \cdot
         \left(
           \prod_{i \in [n]}
             a_{i,\sig_2(i)}
         \right).
  \end{align*}
\end{lemma}

\begin{proof}
  Fix some $\sig_1, \sig_2 \in \Sn$. Let $\hvx \defeq h(\sig_1,\sig_2)$. Then
  \begin{align*}
    \hat{g}\bigl( h(\sig_1,\sig_2) \bigr)
      &= \hat{g}(\hvx) \\
      &= \left(
           \prod_{i \in [n]}
             \hfL{i}(\hat x_{i,1}, \ldots, \hat x_{i,n})
         \right) \\
      &\hspace{0.25cm}
         \cdot
         \left(
           \prod_{j \in [n]}
             \hfR{j}(\hat x_{1,j}, \ldots, \hat x_{n,j})
         \right) \\
      &= \left(
           \prod_{i \in [n]}
             a_{i,\sig_1(i)}
         \right)
         \cdot
         \left(
           \prod_{i \in [n]}
             a_{i,\sig_2(i)}
         \right).
  \end{align*}
  where the second equality follows from the specification of
  $\graph{\hat N}(\matr{A})$ in Proposition~\ref{prop:average:double:cover:1}
  and where the third equality is a consequence of the following observations:
  \begin{itemize}
    
  \item As seen above (see the text before
    Example~\ref{example:valid:configuration:1}), every vertex of
    $\graph{\hat N}(\matr{A})$ is either the endpoint of exactly one
    $(1,1)$-edge or the vertex of exactly one $(0,1)$-cycle.

  \item Consider a $(1,1)$-edge of $\hvx$ connecting $\hfL{i}$ and
    $\hfR{j}$. Because
    \begin{align*}
      \hfL{i}(\hat x_{i,1}, \ldots, \hat x_{i,n})
        &= a_{i,j}, \\
      \hfR{j}(\hat x_{1,j}, \ldots, \hat x_{n,j})
        &= a_{i,j},
    \end{align*}
    we get a contribution of $a_{i,j}^2$ from the two endpoints of this
    $(1,1)$-edge. Note that
    \begin{align*}
      a_{i,j}^2 
        &= a_{i,\sigma_1(i)} \cdot a_{i,\sigma_2(i)}.
    \end{align*}

  \item Consider a $(0,1)$-cycle of length $2L$, $L \geq 2$, of $\hvx$ through
    $\hfL{i_1}, \hfR{j_1}, \hfL{i_2}, \hfR{j_2}, \ldots, \hfL{i_L}, \hfR{j_L},
    \hfL{i_1}$, where
    \begin{alignat*}{3}
      \sig_1(i_1) &= j_1, \ \
      \sig_1(i_2) &= j_2, \ \ 
      \ldots, \ \ 
      \sig_1(i_L) &= j_L, \\
      \sig_2(i_1) &= j_L, \ \ 
      \sig_2(i_2) &= j_1, \ \ 
      \ldots, \ \ 
      \sig_2(i_L) &= j_{L-1}.
    \end{alignat*}
    Because
    \begin{align*}
      \hfL{i_1}(\hat x_{i_1,1}, \ldots, \hat x_{i_1,n})
        &= \sqrt{a_{i_1,j_L} \cdot a_{i_1,j_1}}, \\
      \hfR{j_1}(\hat x_{1,j_1}, \ldots, \hat x_{n,j_1})
        &= \sqrt{a_{i_1,j_1} \cdot a_{i_2,j_1}}, \\
      \hfL{i_2}(\hat x_{i_2,1}, \ldots, \hat x_{i_2,n})
        &= \sqrt{a_{i_2,j_1} \cdot a_{i_2,j_2}}, \\
      \hfR{j_2}(\hat x_{1,j_2}, \ldots, \hat x_{n,j_2})
        &= \sqrt{a_{i_2,j_2} \cdot a_{i_3,j_2}}, \\
      \vdots \hspace{0.5cm}
        &\hspace{1.0cm} 
           \vdots \\
      \hfL{i_L}(\hat x_{i_L,1}, \ldots, \hat x_{i_L,n})
        &= \sqrt{a_{i_L,j_{L-1}} \cdot a_{i_L,j_L}}, \\
      \hfR{j_L}(\hat x_{1,j_L}, \ldots, \hat x_{n,j_L})
        &= \sqrt{a_{i_L,j_L} \cdot a_{i_1,j_L}},
    \end{align*}
    we get a contribution of 
    \begin{align*}
      \left(
        \prod_{\ell=1}^{L}
          a_{i_{\ell},\sigma_1(i_{\ell})}
      \right)
      \cdot
      \left(
        \prod_{\ell=1}^{L}
          a_{i_{\ell},\sigma_2(i_{\ell})}
      \right)
    \end{align*}
    from the $2L$ vertices of this $(0,1)$-cycle.

  \end{itemize}
\end{proof}

Putting together the above results, we can finally prove
Proposition~\ref{sec:proof:prop:average:double:cover:2}. Namely, we obtain
\begin{align*}
  Z(\hat{\graphN}\bigl( \matr{A}) \bigr)
    &= \sum_{\hvx \in \setset{C}(\graph{\hat N}(\matr{A}))}
         \hat{g}(\hvx) \\
    &= \sum_{\hvx \in \setset{C}(\graph{\hat N}(\matr{A}))} \ 
         \sum_{\substack{\sig_1, \sig_2 \in \setset{S}_n: \\ h(\sig_1,\sig_2) = \hvx}}
         2^{-c(\sig_1,\sig_2)}
         \cdot
         \hat{g}\bigl( \hvx \bigr) \\
    &= \sum_{\sig_1, \sig_2 \in \Sn}
         2^{-c(\sig_1,\sig_2)}
         \cdot
         \hat{g}\bigl( h(\sig_1,\sig_2) \bigr) \\
    &= \!\!\!\!\! 
       \sum_{\sig_1, \sig_2 \in \Sn}
       \!\!\!\!\! 
         2^{-c(\sig_1,\sig_2)}
         \! \cdot \!
         \left(
           \prod_{i \in [n]} \!
             a_{i,\sig_1(i)}
         \right)
         \! \cdot \!
         \left(
           \prod_{i \in [n]} \!
             a_{i,\sig_2(i)}
         \right) \! ,
\end{align*}
where the second equality follows from
Lemma~\ref{lemma:number:of:pre:images:1} and where the fourth equality follows
from Lemma~\ref{lemma:global:function:value:1}. This proves the validity
of~\eqref{eq:proposition:reformulation:1}, and with that the validity
of Proposition~\ref{sec:proof:prop:average:double:cover:2}.

\section{Proof of Theorem~\ref{theorem:all:one:matrix:1}}
\label{sec:proof:theorem:all:one:matrix:1}

Let $\group{G}$ be a subgroup of $\Sn$. The cycle index of $\group{G}$
is defined to be (see, e.g., \cite[Section 6.6]{Comtet1974})
\begin{align*}
  Z(\group{G})
   & \defeq
  \frac{1}{\abs{\group{G}}}
  \sum_{\sig \in \group{G}}
  \prod_{k \in [n]} z_k^{c_k(\sig)},
\end{align*}
where $c_k(\sig)$, $k \in [n]$, denotes the number of cycles of length~$k$ in
the cycle notation expression of the permutation $\sig$, and where $z_k$,
$k \in [n]$, are indeterminates. If $\group{G} = \Sn$, which is the case of
interest here, then
\begin{align*}
  Z(\Sn)
   & \defeq
  \frac{1}{\abs{\Sn}}
  \sum_{\sig \in \Sn}
  \prod_{k \in [n]} z_k^{c_k(\sig)} \\
   & = \frac{1}{n!}
  \sum_{\sig \in \Sn}
  \prod_{k \in [n]} z_k^{c_k(\sig)}.
\end{align*}
The following well-known result gives a convenient recursive expression for
$Z(\Sn)$.

\begin{lemma}
  Define $Z(\symgrp{0}) \defeq 1$. For $n \ge 1$ it holds that
  \begin{align}
    Z(\Sn)
     & = \frac{1}{n}
    \sum_{\l \in [n]}
    z_\l
    \cdot
    Z(\symgrp{n - \l}).
    \label{eq:cycle-index:recursion}
  \end{align}
\end{lemma}

\begin{proof}
  For $\l \in [n]$, let $\Sn(\l)$ be subset of $\Sn$ that contains all
  permutations $\sig$ such that $n$ is contained in a cycle of $\sig$ of
  length $\l$. Note that
  $\bigl| \Sn(\l) \bigr| / \bigl| \symgrp{n-\l} \bigr| = \binom{n - 1}{\l - 1}
    \cdot (\l - 1)!$. Therefore,
  \begin{align*}
    Z(\Sn)
     & = \frac{1}{n!}
    \sum_{\sig \in \Sn}
    \prod_{k \in [n]} z_k^{c_k(\sig)}        \\
     & = \frac{1}{n!}
    \sum_{\l \in [n]}
    \sum_{\sig \in \Sn(\l)}
    \prod_{k \in [n]} z_k^{c_k(\sig)}        \\
     & = \frac{1}{n!}
    \sum_{\l \in [n]}
    \binom{n - 1}{\l - 1}
    \cdot
    (\l - 1)!
    \cdot
    z_\l
    \cdot
    \!\!\!
    \sum_{\sig \in \symgrp{n - \l}}
      \prod_{k \in [n-\l]} z_k^{c_k(\sig)} \\
     & = \frac{1}{n}
    \sum_{\l \in [n]}
    z_\l
    \cdot
    Z(\mathcal{S}_{n - \l}).
  \end{align*}
\end{proof}

For $n \in \dN \cup \{ 0 \}$, define $Z_n \defeq Z(\Sn)$, where
\begin{align*}
  z_k
   & \defeq
  \begin{cases}
    1           & \text{if $k = 1$}    \\
    \frac{1}{2} & \text{if $k \geq 2$}
  \end{cases}
\end{align*}
Rewriting~\eqref{eq:cycle-index:recursion} for this choice of $z_k$,
$k \geq 1$, we obtain
\begin{align}
  Z_n
   & = \frac{1}{n}
       Z_{n-1}
       +
       \frac{1}{2n}
       \sum_{\l = 2}^n
         Z_{n-\l}.
  \label{eq:cycle-index:recursion:2}
\end{align}

Contemplating Fig.~\ref{all-1-mat:Zn-vs-inv-sqrt-n-loglog}, it appears
that $Z_n = \Theta(\frac{1}{\sqrt{n}})$. In the following, we will prove this
observation. In fact, we will prove
\begin{align*}
  \dfrac{\Cl}{\sqrt{n}}
   & \leq
  Z_n
  \leq
  \dfrac{\Cu}{\sqrt{n}}, \ n \geq 1,
  \qquad \text{and} \qquad
  Z_n
  \sim
  \frac{C}{\sqrt{n}},
\end{align*}
where
\begin{align}
  \Cl
   & \defeq
  \frac{1}{\sqrt{2}},
  \quad
  \Cu
  \defeq
  \frac{3}{2 \sqrt{2}},
  \quad
  C
  \defeq
  \sqrt{\frac{e}{\pi}}.
  \label{eq:def:const:Cl:Cu:C}
\end{align}

\begin{figure}
  \begin{center}
    \includegraphics[width=\columnwidth]{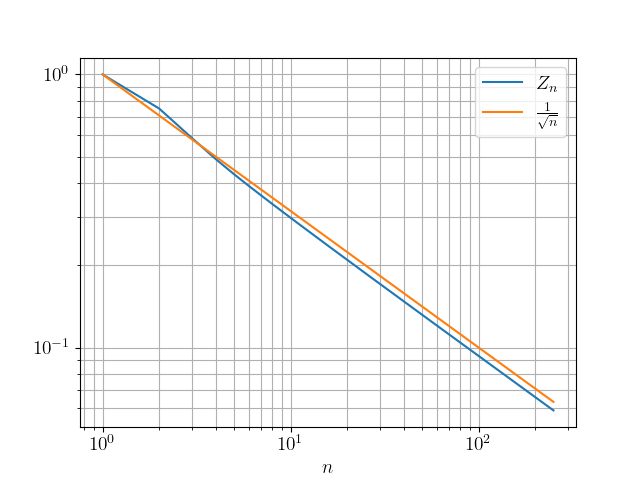}
    \caption{Graph comparing the functions $n \mapsto Z_n$ with $n \mapsto
        \frac{1}{\sqrt{n}}$ for $n \in [250]$.}
    \label{all-1-mat:Zn-vs-inv-sqrt-n-loglog}

    \includegraphics[width=\columnwidth]{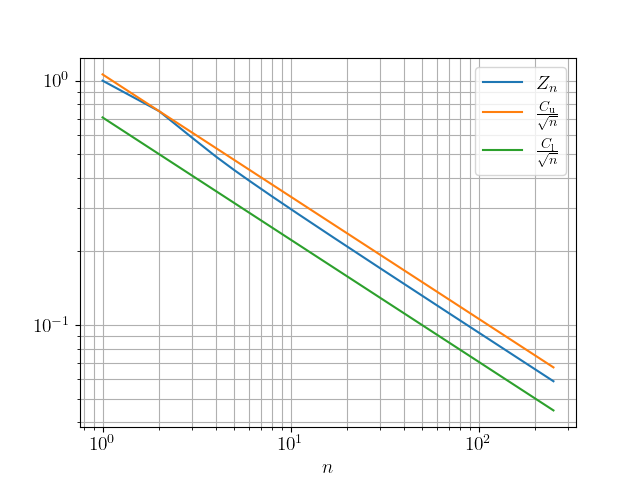}
    \caption{Graph comparing the functions $n \mapsto Z_n$ with
      $n \mapsto \frac{\Cu}{\sqrt{n}}$ and $n \mapsto \frac{\Cl}{\sqrt{n}}$ for
      $n \in [250]$.}
    \label{all-1-mat:Zn-vs-bounds-loglog}
  \end{center}
\end{figure}

\begin{lemma}
  \label{lemma:Zn-ub}
  For all $n \in \dN$ it holds that
  \begin{align*}
    Z_n
     & \leq
    \dfrac{\Cu}{\sqrt{n}},
  \end{align*}
  where $\Cu$ was defined in~\eqref{eq:def:const:Cl:Cu:C}. Thus
  $Z_n = \textnormal{O}\bigl( \frac{1}{\sqrt{n}} \bigr)$.
\end{lemma}

\begin{proof}
  The proof is by strong induction.

  Base cases ($n = 1$ and $n = 2$):
  $Z_1 = 1 \leq \frac{3}{2\sqrt{2}} = \frac{\Cu}{\sqrt{1}}$ and
  $Z_2 = \frac{3}{4} = \frac{3}{2\sqrt{2}} \cdot \frac{1}{\sqrt{2}} =
    \frac{\Cu}{\sqrt{2}}$.

  Induction step ($n = k + 1$ for some $k \in \dN$, $k \geq 2$): Assume the
  claim is true for every $j \in
    [k]$. Using~\eqref{eq:cycle-index:recursion:2}, we obtain
  \begin{pbalign*}
    Z_{k + 1}
    &= \frac{Z_k}{k + 1}
    +
    \frac{1}{2(k + 1)}\sum_{j = 2}^{k + 1} Z_{k + 1 - j} \\
    &= \frac{2 + Z_k}{2(k + 1)}
    +
    \frac{1}{2(k + 1)}\sum_{j = 2}^{k - 1} Z_{k + 1 - j} \\
    &= \frac{2 + Z_k}{2(k + 1)}
    +
    \frac{1}{2(k + 1)} \sum_{j = 2}^k Z_j \\
    &\leq
    \frac{1}{k + 1} \cdot \Big(1 + \frac{\Cu}{2\sqrt{k}}\Big)
    +
    \frac{1}{2(k + 1)}\sum_{j = 2}^k \frac{\Cu}{\sqrt{j}} \\
    &\leq
    \frac{\Cu}{k + 1} \cdot \Big(\frac{1}{\Cu} + \frac{1}{2\sqrt{k}}\Big)
    +
    \frac{\Cu}{2(k + 1)}
    \int_1^k \frac{1}{\sqrt{x}} \operatorname{d}\!x \\
    &= \frac{\Cu}{k + 1}
    \cdot
    \left(
    \frac{1}{\Cu}
    + \frac{1}{2\sqrt{k}}
    + (\sqrt{k} - 1)
    \right) \\
    &\leq
    \frac{\Cu}{k + 1} \cdot \sqrt{k + 1} \\
    &= \frac{\Cu}{\sqrt{k + 1}}.
  \end{pbalign*}
  Notice that the last inequality is valid since the function
  $f: \dRPP \to \dR, \ x \mapsto \sqrt{x + 1} - \sqrt{x} -\frac{1}{2\sqrt{x}} +
    1 - \frac{2\sqrt{2}}{3}$ is positive for $x \geq 2$, which can be proven by
  the observing that $f(2) \approx 0.0215 > 0$ and that $f$ is strictly
  increasing because its derivative satisfies
  $f'(x) = \frac{1}{2\sqrt{x + 1}} - \frac{1}{2\sqrt{x}} + \frac{1}{4x^{3/2}}
    > 0$ for all $x > 0$.
\end{proof}

\begin{lemma}\label{lemma:Zn-lb}
  For all $n \in \dN$ it holds that
  \begin{align*}
    Z_n
     & \geq
    \dfrac{\Cl}{\sqrt{n}},
  \end{align*}
  where $\Cl$ was defined in~\eqref{eq:def:const:Cl:Cu:C}. Thus
  $Z_n = \Omega\bigl( \frac{1}{\sqrt{n}} \bigr)$.
\end{lemma}

\begin{proof}
  The proof is by strong induction.

  Base cases ($n = 1$ and $n = 2$):
  $Z_1 = 1 \geq \frac{1}{\sqrt{2}} = \frac{\Cl}{\sqrt{1}}$ and
  $Z_2 = \frac{3}{4} = \frac{3}{2} \cdot \frac{1}{\sqrt{2}} \cdot
    \frac{1}{\sqrt{2}} \geq \frac{\Cl}{\sqrt{2}}$.

  Induction step ($n = k + 1$ for some $k \in \dN$, $k \geq 2$): Assume the
  claim is true for every $j \in
    [k]$. Using~\eqref{eq:cycle-index:recursion:2}, we obtain
  \begin{align*}
    Z_{k + 1}
     & = \frac{Z_k}{k + 1}
    +
    \frac{1}{2(k + 1)}\sum_{j = 2}^{k + 1} Z_{k + 1 - j}   \\
     & = \frac{2 + Z_k}{2(k + 1)}
    +
    \frac{1}{2(k + 1)}\sum_{j = 2}^{k - 1} Z_{k + 1 - j}   \\
     & \geq
    \frac{1}{k + 1} \cdot \Big(1 + \frac{\Cl}{2\sqrt{k}}\Big)
    +
    \frac{1}{2(k + 1)} \sum_{j = 2}^k \frac{\Cl}{\sqrt{j}} \\
     & \geq
    \frac{\Cl}{k + 1}
    \cdot
    \Big(
    \frac{1}{\Cl} + \frac{1}{2\sqrt{k}}
    \Big)
    +
    \frac{\Cl}{2(k + 1)}
    \int_2^{k + 1}
    \frac{1}{\sqrt{x}} \operatorname{d}\!x                 \\
     & = \frac{\Cl}{k + 1}
    \cdot
    \Big(
    \frac{1}{\Cl} + \frac{1}{2\sqrt{k}} + (\sqrt{k + 1} - \sqrt{2})
    \Big)                                                  \\
     & \geq
    \frac{\Cl}{k + 1} \cdot \sqrt{k + 1}                   \\
     & = \frac{\Cl}{\sqrt{k + 1}}.
  \end{align*}
  Notice that the last inequality is valid since the function
  $f: \dRPP \to \dR, \ x \mapsto \frac{1}{2\sqrt{x}}$ is positive for
  $x \geq 2$.
\end{proof}

\begin{proposition}\label{prop:Zn-asym-func}
  It holds that
  \begin{align*}
    Z_n
     & = \Theta\Big(\dfrac{1}{\sqrt{n}}\Big).
  \end{align*}
\end{proposition}

\begin{proof}
  This follows immediately from \Cref{lemma:Zn-ub,lemma:Zn-lb}.
\end{proof}

A comparison of the function $n \mapsto Z_n$ with the upper and lower bounds
in Lemmas~\ref{lemma:Zn-ub} and~\ref{lemma:Zn-lb} is shown in
Fig.~\ref{all-1-mat:Zn-vs-bounds-loglog}.

Toward proving $Z_n \sim \frac{C}{\sqrt{n}}$, where
$C \defeq \sqrt{\frac{e}{\pi}}$, we use the following ingredients:
\begin{itemize}

  \item It holds that
    \begin{align*}
      Z(\Sn)
       & = \frac{1}{n!}
      \cdot
      B_n\bigl( 0! \cdot z_1, 1! \cdot z_2, \ldots, (n - 1)! \cdot z_n \bigr),
    \end{align*}
    where $B_n$ denotes the $n$-th complete exponential Bell polynomial. (This
    relationship follows from the fact that a permutation is, by definition,
    bijective, and from identifying each monomial in $Z(\Sn)$ with a
    partitioning of $[n]$.) Consequently,
    \begin{align*}
      Z_n
       & = \frac{1}{n!}
      \cdot
      B_n
      \left(
      0! \cdot 1, 1! \cdot \frac{1}{2}, \ldots, (n - 1)! \cdot \frac{1}{2}
      \right).
    \end{align*}

  \item It holds that (see, e.g., \cite[Theorem 4.38]{Wilf2005})
    \begin{align*}
      Z(\Sn)
       & = \left.
      \left(
      \frac{\pa}{\pa t}
      \right)^{\!\! n}
      C(t) \
      \right|_{t = 0},
    \end{align*}
    where
    \begin{align*}
      C(t)
       & \defeq
      \exp
      \left(
      \sum_{k = 1}^\infty
      \frac{z_k t^k}{k}
      \right)
    \end{align*}
    is the generating function of the cycle index of the symmetric group.

  \item We recall the following well-known power series, which are convergent
    in a neighborhood around $t = 0$:
    \begin{align*}
      \ln (1 - t)
       & = -
      \sum_{k = 1}^\infty
      \frac{t^k}{k},           \\
      \e^t
       & = \sum_{k = 0}^\infty
      \frac{t^k}{k!},          \\
      (1 - t)^{-1/2}
       & = \sum_{k = 0}^\infty
      \frac{(2k)!}{4^k \cdot (k!)^2} t^k.
    \end{align*}

  \item We recall the following bounds on the $k$-th central binomial
    coefficient for $k \in \dN$:
    \begin{align}
      \label{eq:cbc-approx}
      \frac{4^k}{\sqrt{(k + \frac{1}{2})\pi}}
       & \leq \binom{2k}{k} \leq \frac{4^k}{\sqrt{k\pi}}.
    \end{align}

\end{itemize}

\begin{lemma}\label{lemma:Zn-asym-lb}
  It holds that
  \[
    \lim_{n \to \infty}Z_n \cdot \sqrt{n} \geq C.
  \]
\end{lemma}

\begin{proof}
  Substituting
  \begin{align*}
    z_k
     & \defeq
    \begin{cases}
      1           & \text{if $k = 1$}    \\
      \frac{1}{2} & \text{if $k \geq 2$}
    \end{cases}
  \end{align*}
  into $C(t)$, we get
  \begin{align*}
    C(t)
     & = \exp\left(t + \frac{1}{2}\sum_{k = 2}^\infty \frac{t^k}{k}\right) \\
     & = \exp\left(t + \frac{1}{2}(-\ln(1 - t) - t)\right)                 \\
     & = \exp\left(\frac{1}{2}(t -\ln(1 - t))\right)                       \\
     & = \e^{t/2} (1 - t)^{-1/2}                                           \\
     & = \left(
    \sum_{k = 0}^\infty \frac{(\frac{t}{2})^k}{k!}
    \right)
    \cdot
    \left(
    \sum_{k = 0}^\infty \frac{(2k)!}{4^k \cdot (k!)^2}t^k
    \right)                                                                \\
     & = \sum_{n = 0}^\infty
    \sum_{\l = 0}^n
    \frac{1}{2^\l \cdot \l!}
    \cdot
    \frac{1}{4^{n - \l}}
    \cdot
    \binom{2(n - \l)}{n - \l}
    \cdot
    t^n.
  \end{align*}
  Letting $\tau \defeq \frac{1}{2}$, we obtain for all $n \in \dN$
  \begin{align*}
    Z_n
     & = \sum_{\l = 0}^n
    \frac{\tau^\l}{\l!} \cdot \frac{1}{4^{n - \l}}
    \cdot
    \binom{2(n - \l)}{n - \l}                           \\
     & \overset{\mathrm{(a)}}{\geq}
    \frac{\tau^n}{n!}
    +
    \sum_{\l = 0}^{n - 1}
    \frac{\tau^\l}{\l!}
    \cdot
    \frac{1}{4^{n - \l}}
    \cdot
    \frac{4^{n - \l}}{\sqrt{(n - \l + \frac{1}{2})\pi}} \\
     & \geq
    \frac{\tau^n}{n!}
    +
    \sum_{\l = 0}^{n - 1}
    \frac{\tau^\l}{\l!}
    \cdot
    \frac{1}{\sqrt{(n + \frac{1}{2})\pi}}               \\
     & \geq
    \frac{1}{\sqrt{n}}
    \cdot
    \frac{1}{\sqrt{(1 + \frac{1}{2n})\pi}}
    \sum_{\l = 0}^n
    \frac{\tau^\l}{\l!},
  \end{align*}
  where step~$\mathrm{(a)}$ is due to \eqref{eq:cbc-approx}. It follows that
  \begin{align*}
    \lim_{n \to \infty} Z_n \! \cdot \! \sqrt{n}
     & \geq
    \lim_{n \to \infty} \! \frac{1}{\sqrt{(1 + \frac{1}{2n})\pi}}
    \sum_{\l = 0}^n \frac{\tau^\l}{\l!}
    = \frac{\exp(\tau)}{\sqrt{\pi}}
    = \sqrt{\frac{e}{\pi}}.
  \end{align*}
\end{proof}

\begin{lemma}\label{lemma:Zn-asym-ub}
  It holds that
  \begin{align*}
    \lim_{n \to \infty}Z_n \cdot \sqrt{n} \leq C.
  \end{align*}
\end{lemma}

\begin{proof}
  The proof is similar to the proof of \Cref{lemma:Zn-asym-lb}. In particular,
  from the proof of \Cref{lemma:Zn-asym-lb} we know that for all $n \in \dN$
  it holds that
  \begin{align*}
    Z_n
     & = \sum_{\l = 0}^n
    \frac{\tau^\l}{\l!}
    \cdot
    \frac{1}{4^{n - \l}}
    \cdot
    \binom{2(n - \l)}{n - \l}.
  \end{align*}
  Using \eqref{eq:cbc-approx}, we obtain
  \begin{align*}
    Z_n
     & \leq
    \frac{\tau^n}{n!}
    +
    \sum_{\l = 0}^{n - 1}
    \frac{\tau^\l}{\l!}
    \cdot
    \frac{1}{4^{n - \l}}
    \cdot
    \frac{4^{n - \l}}{\sqrt{(n - \l)\pi}} \\
     & = \frac{\tau^n}{n!}
    +
    \frac{1}{\sqrt{\pi}}
    \sum_{\l = 0}^{n - 1}
    \frac{\tau^\l}{\l!}
    \cdot
    \frac{1}{\sqrt{n - \l}} \\
     & = \frac{\tau^n}{n!}
    +
    \frac{1}{\sqrt{\pi}} \!
    \left(
    \sum_{\l = 0}^{h(n)}
    \frac{\tau^\l}{\l!}
    \! \cdot \!
    \frac{1}{\sqrt{n - \l}}
    + \!\!\!
    \sum_{\l = h(n) + 1}^{n - 1}
    \frac{\tau^\l}{\l!}
    \! \cdot \!
    \frac{1}{\sqrt{n - \l}}
    \right) \\
     & \leq
    \frac{\tau^n}{n!}
    +
    \frac{1}{\sqrt{\pi}} \!
    \left(
      \sum_{\l = 0}^{h(n)} \!
      \frac{\tau^\l}{\l!}
      \! \cdot \!
      \frac{1}{\sqrt{n - h(n)}}
      \! + \!\!\!\!\!
      \sum_{\l = h(n) + 1}^{n - 1} \!\!
      \frac{\tau^\l}{\l!}
      \! \cdot \!
      \frac{1}{\sqrt{n - \l}}
    \right) \\
     & =
    \frac{\tau^n}{n!}
    +
    \frac{(1 - \frac{h(n)}{n})^{-1/2}}
    {\sqrt{n}
      \cdot
      \sqrt{\pi}}
    \sum_{\l = 0}^{h(n)}
    \frac{\tau^\l}{\l!}
    +
    \sum_{\l = h(n) + 1}^n
    \frac{\tau^\l}{\l!}
    \cdot
    \frac{1}{\sqrt{n - \l}},
  \end{align*}
  where $h: \dN \to \dN$ is a function satisfying
  $h(n) = \omega_n\bigl( 1 \bigr)$ and $h(n) = o(n)$. In particular, there exists an $N \in \dN$ such that $\forall\ n \geq N$, it holds that $1 < h(n) < n$. Then notice that
  \begin{align}
    0
     & \leq
    \sum_{\l = h(n) + 1}^{n - 1}
    \frac{\tau^\l}
    {\l!}
    \cdot
    \frac{\sqrt{n}}{\sqrt{n - \l}} \nonumber \\
     & \overset{\mathrm{(a)}}{\leq}
    \sum_{\l = h(n) + 1}^{n - 1}
    \tau^\l
    \cdot
    \frac{\sqrt{n}}{\sqrt{n - \l}}
    \cdot
    \frac{1}{\sqrt{2\pi\l}}
    \left( \frac{\e}{\l} \right)^\l
    \e^{-1/(12\l + 1)} \nonumber \\
     & \leq
    \frac{1}{\sqrt{2\pi}}
    \sum_{\l = h(n) + 1}^{n - 1}
    \frac{\sqrt{n}}{\sqrt{\l(n - \l)}}
    \left(\frac{\tau\e}{\l}\right)^\l
    \e^{-0} \nonumber \\
     & \overset{\mathrm{(b)}}{\leq}
    \frac{1}{\sqrt{2\pi}}
    \sum_{\l = h(n) + 1}^{n - 1}
    \frac{\sqrt{n}}{\sqrt{(n - 1)(n - (n - 1))}}
    \left(\frac{\tau\e}{\l}\right)^\l \nonumber \\
     & \leq
    \frac{1}{\sqrt{2\pi}}
    \left( 1 - \frac{1}{n} \right)^{-1/2}
    \sum_{\l = h(n) + 1}^{n - 1}
    \left(\frac{\tau\e}{h(n)}\right)^\l \nonumber \\
     & =
    \frac{1}{\sqrt{2\pi}}
    \left( 1 - \frac{1}{n} \right)^{-1/2}
    \frac{
      \left( \frac{\tau\e}{h(n)} \right)^{h(n) + 1}
      -
      \left( \frac{\tau\e}{h(n)} \right)^n}
    {1 - \frac{\tau\e}{h(n)}}
    , \label{eq:Zn-ub}
  \end{align}
  where step~$\mathrm{(a)}$ is due to Robbins' approximations for factorial
  function \cite{10.2307/2308012}, and where step~$\mathrm{(b)}$ is due to the
  observation that $x = n - 1$ is a global minimum for the polynomial function
  $x \mapsto x \cdot (n \! - \! x)$ on $[n \! - \! 1]$.

  Note that the limit of the right-hand side of \eqref{eq:Zn-ub} equals $0$ as
  $n \to \infty$ since $\frac{\tau\e}{h(n)} < 1$ for sufficiently large $n$,
  which is guaranteed by $h(n) = \omega_n(1)$. So, we obtain
  \begin{align*}
     & \hspace{-0.35cm}
    \lim_{n \to \infty} Z_n \cdot \sqrt{n} \\
     & \leq
    \lim_{n \to \infty}
    \biggl(
    \frac{\sqrt{n} \cdot \tau^n}{n!}
    +
    \frac{(1 - \frac{h(n)}{n})^{-1/2}}
      {\sqrt{\pi}}
    \sum_{\l = 0}^{h(n)}
    \frac{\tau^\l}
    {\l!}                                  \\
     & \hspace{0.45\columnwidth}
    +
    \sum_{\l = h(n) + 1}^n
    \frac{\tau^\l}{\l!}
    \cdot
    \frac{\sqrt{n}}{\sqrt{n - \l}}
    \biggl)                                \\
     & = 0
    +
    \frac{(1 - 0)^{-1/2}}{\sqrt{\pi}}
    \cdot
    \exp\left( \frac{1}{2} \right)
    +
    0 = \sqrt{\frac{e}{\pi}}.
  \end{align*}
\end{proof}

\begin{proposition}
  It holds that
  \begin{align*}
    Z_n
     & \sim
    C
    \cdot
    \frac{1}{\sqrt{n}},
  \end{align*}
  where $C$ was defined in~\eqref{eq:def:const:Cl:Cu:C}.
\end{proposition}

\begin{proof}
  This follows directly from \Cref{lemma:Zn-asym-lb,lemma:Zn-asym-ub}.
\end{proof}

\section{Proof of Theorem~\ref{theorem:iid:matrices:1}}
\label{sec:proof:theorem:iid:matrices:1}

Given Assumption~\ref{ass:iid:matrices:1}, we are interested in the
expectation value of $\perm(\matr{A})$, $\perm(\matr{A})^2$, and
$\perm_{\Bethe,2}(\matr{A})^2$.

The expectation value of the permanent of $\matr{A}$ is given by\footnote{Note
  that we actually do not need this result in the following.}
\begin{align*}
  \expvalbig{\perm(\matr{A})}
   & = \expval{
    \sum_{\sig \in \Sn}
    \prod_{i \in [n]}
    a_{i, \sig(i)}
  }                            \\
   & \overset{\mathrm{(a)}}{=}
  \sum_{\sig \in \Sn}
  \expval{
    \prod_{i \in [n]}
    a_{i, \sig(i)}
  }                            \\
   & \overset{\mathrm{(b)}}{=}
  \sum_{\sig \in \Sn}
  \prod_{i \in [n]}
  \expval{a_{i, \sig(i)}}      \\
   & = n! \cdot \mu_1^n,
\end{align*}
where step~$\mathrm{(a)}$ is due to the linearity of the expectation value, and where step~$\mathrm{(b)}$ is due to the independence between the entries of $\matr{A}$. Similarly, we obtain
\begin{align*}
  \expvalbig{\perm(\matr{A})^2}
   & = \sum_{\sig_1, \sig_2 \in \Sn}
  \prod_{i \in [n]}
  \expval{a_{i, \sig_1(i)} \cdot a_{i, \sig_2(i)}} \\
   & = \sum_{\sig_1, \sig_2 \in \Sn}
  \prod_{i \in [n]}
  h_i(\sig_1,\sig_2)
\end{align*}
where
\begin{align*}
  h_i(\sig_1,\sig_2)
   & \defeq
  \expval{a_{i, \sig_1(i)} \cdot a_{i, \sig_2(i)}}
  = \begin{cases}
    \mu_2   & \text{if} \ \sig_1(i) = \sig_2(i) \\
    \mu_1^2 & \text{otherwise}
  \end{cases}.
\end{align*}
Because
\begin{align*}
  h_i(\sig_1,\sig_2)
    &= h_i(\sig_1 \circ \sig_2^{-1}, \sig_2 \circ \sig_2^{-1}) 
     = h_i( \sig_1\circ \sig_2^{-1}, \mathrm{id}),
\end{align*}
where $\mathrm{id}$ is the identity permutation, we obtain
\begin{align*}
  \expvalbig{\perm(\matr{A})^2}
   & = \sum_{\sig_1, \sig_2 \in \Sn}
  \prod_{i \in [n]}
  h_i(\sig_1,\sig_2)                        \\
   & = \sum_{\sig_1, \sig_2 \in \Sn}
  \prod_{i \in [n]}
  h_i(\sig_1 \circ \sig_2^{-1},\mathrm{id}) \\
   & = n!
  \cdot
  \sum_{\sig \in \Sn}
  \prod_{i \in [n]}
  h_i(\sig,\mathrm{id}),
\end{align*}
where
\begin{align*}
  h_i(\sig,\mathrm{id})
   & = \begin{cases}
    \mu_2   & \text{if} \ \sig(i) = i \\
    \mu_1^2 & \text{otherwise}
  \end{cases}.
\end{align*}
Note that every cycle of length $1$ of $\sig$ contributes a factor of $\mu_2$
to $\prod_{i \in [n]} h_i(\sig,\mathrm{id})$. Similarly, every cycle of length
$k$, $k \geq 2$, contributes a factor of $(\mu_1^2)^k$ to
$\prod_{i \in [n]} h_i(\sig,\mathrm{id})$. Therefore, using the cycle index of
$\Sn$, we can write
\begin{align*}
  \expvalbig{\perm(\matr{A})^2}
   & = Z(\Sn),
\end{align*}
where
\begin{align*}
  z_k
   & = \begin{cases}
    \mu_2       & \text{if $k = 1$}    \\
    (\mu_1^2)^k & \text{if $k \geq 2$}
  \end{cases}.
\end{align*}
Applying the generation function technique as in the proofs of
\Cref{lemma:Zn-asym-lb,lemma:Zn-asym-ub}, we obtain
\begin{align}
  \expvalbig{\perm(\matr{A})^2}
   & = (n!)^2
  \cdot
  \sum_{\l = 0}^n
  \frac{(\mu_2 - \mu_1^2)^\l
  \mu_1^{2(n - \l)}}{\l!} \nonumber \\
   & = (n!)^2
  \cdot
  \mu_1^{2n}
  \cdot
  \sum_{\l = 0}^n
  \frac{(\frac{\mu_2}{\mu_1^2} - 1)^\l}
  {\l!} \nonumber                   \\
   & \sim
  (n!)^2
  \cdot
  \mu_1^{2n}
  \cdot
  \exp
  \left(
  \frac{\mu_2}{\mu_1^2} - 1
  \right).
  \label{eq:expval-perm-square}
\end{align}
Similarly, one obtains
\begin{align}
  \label{eq:expval-b2-perm-square}
   & \hspace{-0.5cm}
  \expvalbig{\perm_{\Bethe, 2}(\matr{A})^2} \nonumber \\
   & = (n!)^2
  \cdot
  \mu_1^{2n}
  \cdot
  \sum_{\l = 0}^n
  \frac{(\frac{\mu_2}{\mu_1^2} - \frac{1}{2})^\l}
  {\l!}
  \cdot
  \frac{(2(n - \l))!}{4^{n - \l} \cdot ((n - \l)!)^2}.
\end{align}
By carefully reusing some calculations in the proofs of
\Cref{lemma:Zn-asym-lb,lemma:Zn-asym-ub}, the summation appearing in
\eqref{eq:expval-b2-perm-square} can then be bounded between
\begin{align*}
  \frac{\mu_1^{2n}}{\sqrt{(n + \frac{1}{2}) \pi}}
  \cdot
  \sum_{\l = 0}^n
  \frac{\tau^\l}{\l!}
\end{align*}
and
\begin{align*}
  \mu_1^{2n}
  \left(
    \frac{\tau^n}{n!}
    \! + \!
    \frac{(1 - \frac{h(n)}{n})^{-1/2}}{\sqrt{n \pi}}
    \cdot
    \sum_{\l = 0}^{h(n)}
    \frac{\tau^\l}{\l!}
    + \!\!\!\!
    \sum_{\l = h(n) + 1}^{n - 1} \!
    \frac{\tau^\l}{\l!}
    \cdot
    \frac{1}{\sqrt{n - \l}}
    \right) \! ,
\end{align*}
where $\tau \defeq \frac{\mu_2}{\mu_1^2} - \frac{1}{2}$. Finally, we obtain
\begin{align}
  \frac{\expvalbig{\perm_{\Bethe, 2}(\matr{A})^2}}
  {(n!)^2}
   & \sim
  \frac{\e^{\tau}}{\sqrt{\pi}}
  \cdot
  \frac{\mu_1^{2n}}{\sqrt{n}}.
  \label{eq:expval-b2-perm-square:2}
\end{align}
Combining~\eqref{eq:expval-perm-square}
and~\eqref{eq:expval-b2-perm-square:2}, we get
\begin{align*}
  \frac{\expvalbig{\perm(\matr{A})^2}}
  {\expvalbig{\perm_{\Bethe, 2}(\matr{A})^2}}
   &\sim
      \frac
      {
        (n!)^2
        \cdot
        \mu_1^{2n}
        \cdot
        \exp
        \left(
          \frac{\mu_2}{\mu_1^2} - 1
        \right)
      }
      {
        (n!)^2
        \cdot
        \frac{\e^{\tau}}{\sqrt{\pi}}
        \cdot
        \frac{\mu_1^{2n}}{\sqrt{n}}
      } \\
   &= \sqrt{\pi n}
      \cdot
      \exp
        \left(
          \tau - \frac{1}{2} - \tau
        \right) \\
   &= \sqrt{\frac{\pi n}{\e}},
\end{align*}
which, surprisingly, does not depend on the chosen distribution!

We remark that the results in
Appendix~\ref{sec:proof:theorem:all:one:matrix:1} can be regarded as a special
case of the results in this section. Indeed, if the entries of $\matr{A}$ are
chosen i.i.d.\ according to the (degenerate) distribution that assigns
$a_{i,j} = 1$ with probability $1$, then $\mu_1 = 1$ and $\mu_2 = 1$, and with
that the right-hand side of~ \eqref{eq:expval-b2-perm-square:2} simplifies to
$\sqrt{\frac{\e}{\pi}} \cdot \frac{1}{\sqrt{n}}$.

\section{Unification of the results in
  Appendices~\ref{sec:proof:theorem:all:one:matrix:1}
  and~\ref{sec:proof:theorem:iid:matrices:1}}
\label{sec:unification:1}

In this appendix, we show that the results in
Appendices~\ref{sec:proof:theorem:all:one:matrix:1}
and~\ref{sec:proof:theorem:iid:matrices:1} can be obtained as special cases of
a more general result.

Recall that for $n \in \dN$, we defined $[n] \defeq \{ 1, 2, \ldots, n
\}$. More generally, for $n_1, n_2 \in \dZ$ such that $n_1 \leq n_2$, we
define $\intset{n_1}{n_2} \defeq \{ n_1, n_1 \! + \! 1, \ldots, n_2 \}$.

\begin{definition}
  \label{def:psi-func}

  Let $n \in \dN$. We define the function
  \begin{alignat*}{3}
    \Psi_n : \
     & \dRP^3            &  & \to     &  & \ \dRNN, \\
     & (\ta_1, \ta_2, \ta_3) &  & \mapsto &  & \
    (n!)^2
    \! \cdot \!
    \ta_2^n
    \! \cdot \!
    \sum_{\l = 0}^n
    \binom{n \! - \! \l \! + \! \ta_3 \! - \! 1}{n \! - \! \l}
    \! \cdot \!
    \frac{(\frac{\ta_1}{\ta_2} \! - \! \ta_3)^\l}{\l!},
  \end{alignat*}
  where the binomial coefficient here is the generalized binomial coefficient,
  which, for $\alpha \in \dR$ and $k \in \dN \cup \{ 0 \}$, is defined to be
  \begin{align*}
    \binom{\alpha}{k}
     & \defeq
    \frac{\alpha \cdot (\alpha - 1) \cdots (\alpha - k + 1)}
    {k!}.
  \end{align*}
\end{definition}

\begin{lemma}
  \label{lemma:psi-func-prop}

  For all $n \in \dN$ it holds that \\[-0.25cm]
  \begin{enumerate}

  \item
    $\Psi_n(\ta_1, \ta_2, \ta_3) \! = \! n! \cdot B_n(0! {\cdot} \ta_1, 1!
    {\cdot} \ta_3\ta_2^2, \ldots, (n \! - \! 1)! {\cdot} \ta_3\ta_2^n)$,
    \\

  \item
    $\Psi_n(\ta_1, \ta_2, \ta_3) \sim (n!)^2 \cdot \ta_2^n \cdot
    \dps\frac{n^{\ta_3 - 1}}{\Ga(\ta_3)} \cdot \exp\left(\frac{\ta_1}{\ta_2} -
      \ta_3\right)$, \\

  \end{enumerate}
  where Item~1 holds for $\ta_3 \in \dR_{>0}$, Item~2 holds for
  $\ta_3 \in \ocintvl{0}{1}$ (possibly for $\ta_3 \in \dR_{>0}$, but we did
  not need this generalization), and where $\Ga(\cdot)$ denotes the Gamma
  function.
\end{lemma}

\begin{proof}
  Item~$1$ follows from analyzing $C(t)$, the generating function of the
  cycle index of the symmetric group, with the substitution
  \begin{align*}
    z_k
     &= \begin{cases}
      \ta_1         & \text{if $k = 1$}    \\
      \ta_3 \ta_2^k & \text{if $k \geq 2$}
    \end{cases}.
  \end{align*}
  In the following, we will use the binomial series, i.e.,
  \begin{align*}
    (1 + t)^\al = \sum_{k = 0}^\infty \binom{\al}{k} t^k,
  \end{align*}
  which is convergent for any $\al \in \dR$ and $t \in \oointvl{-1}{1}$. (Note
  that the binomial coefficient here is the generalized binomial coefficient.)

  Adapting some calculations in the proof of \Cref{lemma:Zn-asym-lb}, we obtain
  \begin{align*}
    C(t)
     & = \exp\left(\ta_1t + \sum_{k = 2}^\infty \frac{\ta_3\ta_2^k}{k} t^k\right) \\
     & = \exp(\ta_1t + \ta_3(-\ln(1 - \ta_2t) - \ta_2t))                          \\
     & = \e^{(\ta_1 - \ta_3\ta_2)t} (1 - \ta_2t)^{-\ta_3}                             \\
     & = \left(\sum_{k = 0}^\infty \frac{\bigl( (\ta_1 - \ta_3\ta_2) t \bigr)^k}{k!}\right)
    \cdot
    \left(\sum_{k = 0}^\infty \binom{-\ta_3}{k}(-\ta_2t)^k\right)                 \\
     & = \left(\sum_{k = 0}^\infty \frac{(\ta_1 - \ta_3\ta_2)^k}{k!} t^k\right)
    \cdot
    \left(\sum_{k = 0}^\infty \ta_2^k\binom{k + \ta_3 - 1}{k} t^k\right)          \\
     & = \sum_{n = 0}^\infty
    \sum_{\l = 0}^n
    \frac{\ta_2^{n - \l}(\ta_1 - \ta_3\ta_2)^\l}{\l!}
    \cdot
    \binom{n - \l + \ta_3 - 1}{n - \l}
    \cdot
    t^n.
  \end{align*}
  Then we can verify that
  \begin{align*}
     & \hspace{-0.5cm}
    n! \cdot B_n(0! \cdot \ta_1, 1! \cdot \ta_3\ta_2^2, \ldots, (n - 1)!
    \cdot
    \ta_3\ta_2^n)                                \\
     & = (n!)^2 \ta_2^n \sum_{\l = 0}^n
    \binom{n - \l + \ta_3 - 1}{n - \l}
    \cdot
    \frac{(\frac{\ta_1}{\ta_2} - \ta_3)^\l}{\l!} \\
     & = \Psi_n(\ta_1, \ta_2, \ta_3).
  \end{align*}

  For Item $2$, we first notice that for any $n \in \dN$,
  $\l \in \intset{0}{n \! - \! 1}$, $\ta_3 \in \ocintvl{0}{1}$, we have
  \begin{align*}
     & \hspace{-0.25cm}
    \binom{n - \l + \ta_3 - 1}{n - \l}                          \\
     & = \frac{\Ga((n - \l + \ta_3 - 1) + 1)}
    {\Ga((n - \l) + 1)\Ga((n - \l + \ta_3 - 1) - (n - \l) + 1)} \\
     & = \frac{\Ga(n - \l + \ta_3)}{\Ga(n - \l + 1)\Ga(\ta_3)}.
  \end{align*}
  In the following, we will need Wendel's double inequalities for the ratio of
  two Gamma function values~\cite{10.2307/2304460}:
  \begin{align}
    \label{eq:wendel-dbl-ineq}
    \left(\frac{x}{x + s}\right)^{1 - s}
     & \leq
    \frac{\Ga(x + s)}{x^s \Ga(x)}
    \leq 1,
  \end{align}
  where $x > 0 \in \dR$, $s \in \ccintvl{0}{1}$. This double inequalities can
  be rewritten as
  \begin{align}
    \label{eq:wendel-dbl-ineq:variant}
    \left( \frac{x}{x + s} \right)^{1 - s}
     & \leq
    \frac{\Ga(x + s)}{x^{s - 1}\Ga(x + 1)}
    \leq 1,
  \end{align}
  or as
  \begin{align}
    \label{eq:wendel-dbl-ineq:variant:2}
    (x + s)^{s - 1}
     & \leq
    \frac{\Ga(x + s)}{\Ga(x + 1)}
    \leq
    x^{s - 1}.
  \end{align}
  Notice that \eqref{eq:wendel-dbl-ineq:variant} is due to the property
  $\Ga(x + 1) = x \cdot \Ga(x)$ of the Gamma function. Substituting
  $x = n - \l$ and $s = c$ into~\eqref{eq:wendel-dbl-ineq:variant:2}, we obtain
  \begin{align*}
    (n - \l + \ta_3)^{\ta_3 - 1}
     & \leq
    \frac{\Ga(n - \l + \ta_3)}{\Ga(n - \l + 1)}
    \leq
    (n - \l)^{\ta_3 - 1}.
  \end{align*}
  Moreover,
  \begin{align*}
    \frac{(n - n + \ta_3)^{\ta_3 - 1}}{\Ga(\ta_3)}
     & = \frac{\Ga(1)}{\ta_3^{1 - \ta_3}\Ga(\ta_3)}
    = \frac{\Ga(c + (1 - \ta_3))}{\ta_3^{1 - \ta_3}\Ga(\ta_3)}
    \leq
    1,
  \end{align*}
  where the inequality is obtained from substituting $x = \ta_3$ and
  $s = 1 - \ta_3$ into \eqref{eq:wendel-dbl-ineq}. Observe that, because $\Ga$
  is decreasing in $\ocintvl{0}{1}$ and because $\Ga(1) = 1$, it holds that
  $\frac{1}{\Ga(\ta_3)} \leq \frac{1}{\Ga(1)} = 1$. Then, letting
  $\tau \defeq \frac{\ta_1}{\ta_2} - \ta_3$, one obtains the expression
  in~\eqref{eq:Psi:n:1} (see top of the next page) and the expression
  in~\eqref{eq:Psi:n:2} (see top of the next page), where $h: \dN \to \dN$ is
  a function satisfying $h(n) = \omega_n\bigl( 1 \bigr)$ and $h(n) = o(n)$. In
  summary, we obtain the expression in~\eqref{eq:Psi:n:3} (see top of the next
  page).

  Finally, by using parts of the proof of \Cref{lemma:Zn-asym-ub}, we obtain
  \begin{align*}
    &\hspace{-0.35cm}
     \sum_{\l = h(n) + 1}^{n - 1}
       \frac{n^{1 - \ta_3}}{(n - \l)^{1 - \ta_3}}
       \cdot
       \frac{\tau^\l}{\l!} \\
    &\leq
       \frac{1}{\sqrt{2\pi}}
       \sum_{\l = h(n) + 1}^{n - 1}
         \frac{n^{1 - \ta_3}}{(n - \l)^{1 - \ta_3}}
         \cdot
         \frac{1}{\sqrt{\l}}
         \left( \frac{\tau\e}{\l} \right)^\l
         \cdot
         \e^{-0} \\
    &\leq
       \frac{1}{\sqrt{2\pi}}
       \sum_{\l = h(n) + 1}^{n - 1}
         \left( \frac{n}{\l(n - \l)} \right)^{1 - \ta_3}
         \cdot
         \left( \frac{\tau\e}{h(n)} \right)^\l \\
    &\leq
        \frac{1}{\sqrt{2\pi}}
        \cdot
        \left( 1 - \frac{1}{n} \right)^{-(1 - \ta_3)}
        \cdot
        \frac{
          \left( \frac{\tau\e}{h(n)} \right)^{h(n) + 1}
          -
          \left( \frac{\tau\e}{h(n)} \right)^n
        }
        {1 - \frac{\tau\e}{h(n)}}
  \end{align*}
  for any $\ta_3 \in \ccintvl{\frac{1}{2}}{1}$ since
  $\l^{-1/2} \leq \l^{-(1 - \ta_3)}$ and
  \begin{pbalign*}
    &\hspace{-0.35cm}
     \sum_{\l = h(n) + 1}^{n - 1}
       \frac{n^{1 - \ta_3}}{(n - \l)^{1 - \ta_3}}
       \cdot
       \frac{\tau^\l}{\l!} \\
    &\leq
       \frac{1}{\sqrt{2\pi}}
       \sum_{\l = h(n) + 1}^{n - 1}
         \frac{n^{1 - \ta_3}}{(n - \l)^{1 - \ta_3}}
         \cdot
         \frac{1}{\sqrt{\l}}
        \left( \frac{\tau\e}{\l} \right)^\l \\
    &= \frac{1}{\sqrt{2\pi}}
         \sum_{\l = h(n) + 1}^{n - 1}
           \left( \frac{n}{\l(n - \l)} \right)^{1 - \ta_3}
           \cdot
           \frac{\tau\e}{\l^{1/2 + \ta_3}}
           \left( \frac{\tau\e}{\l} \right)^{\l - 1} \\
    &\leq
       \frac{\tau\e}{\sqrt{2\pi}}
       \cdot
       \left( 1 - \frac{1}{n} \right)^{-(1 - \ta_3)}
       \frac{
         \left( \frac{\tau\e}{h(n)} \right)^{h(n)}
         -
         \left( \frac{\tau\e}{h(n)} \right)^{n - 1}
       }
       {1 - \frac{\tau\e}{h(n)}}
  \end{pbalign*}
  for any $\ta_3 \in \oointvl{0}{\frac{1}{2}}$.
  The desired result is then obtained by letting $n \to \infty$ and
  applying the sandwich theorem.
\end{proof}

\begin{figure*}[t]
  \begin{align}
    \Psi_n(\ta_1, \ta_2, \ta_3)
     & = (n!)^2
    \cdot
    \ta_2^n
    \cdot
    \frac{\tau^n}{n!}
    +
    (n!)^2
    \cdot
    \ta_2^n
    \cdot
    \sum_{\l = 0}^{n - 1}
    \frac{\Ga(n - \l + \ta_3)}
    {\Ga(n - \l + 1)\Ga(\ta_3)}
    \cdot
    \frac{\tau^\l}{\l!}
    \nonumber                     \\
     & \geq
    (n!)^2
    \cdot
    \ta_2^n
    \cdot
    \frac{\tau^n}{n!}
    \cdot
    \frac{(n - n + \ta_3)^{\ta_3 - 1}}
    {\Ga(\ta_3)}
    +
    (n!)^2
    \cdot
    \ta_2^n
    \sum_{\l = 0}^{n - 1}
    \frac{(n - \l + \ta_3)^{\ta_3 - 1}}
    {\Ga(\ta_3)}
    \cdot
    \frac{\tau^\l}{\l!}
    \nonumber                     \\
     & \geq
    (n!)^2
    \cdot
    \ta_2^n
    \cdot
    \frac{(n + \ta_3)^{\ta_3 - 1}}
    {\Ga(\ta_3)}
    \sum_{\l = 0}^n
    \frac{\tau^\l}{\l!} \nonumber \\
     & = (n!)^2
    \cdot
    \ta_2^n
    \cdot
    \frac{n^{\ta_3 - 1}}{\Ga(\ta_3)}
    \cdot
    \left(
    1
    +
    \frac{\ta_3}{n}
    \right)^{\ta_3 - 1}
    \cdot
    \sum_{\l = 0}^n \frac{\tau^\l}{\l!},
    \label{eq:Psi:n:1}            \\
    \Psi_n(\ta_1, \ta_2, \ta_3)
     & \leq
    (n!)^2
    \cdot
    \ta_2^n
    \cdot
    \frac{\tau^n}{n!}
    +
    (n!)^2
    \cdot
    \ta_2^n
    \cdot
    \sum_{\l = 0}^{n - 1}
    \frac{(n - \l)^{\ta_3 - 1}}{\Ga(\ta_3)}
    \cdot
    \frac{\tau^\l}{\l!}
    \nonumber                     \\
     & = (n!)^2
    \cdot
    \ta_2^n
    \cdot
    \frac{\tau^n}{n!}
    +
    (n!)^2
    \cdot
    \ta_2^n
    \cdot
    \left(
    \sum_{\l = 0}^{h(n)}
    \frac{(n - \l)^{\ta_3 - 1}}{\Ga(\ta_3)}
    \cdot
    \frac{\tau^\l}{\l!}
    +
    \sum_{\l = h(n) + 1}^{n - 1}
    \frac{(n - \l)^{\ta_3 - 1}}{\Ga(\ta_3)}
    \cdot
    \frac{\tau^\l}{\l!}
    \right)
    \nonumber                     \\
     & \leq
    (n!)^2
    \cdot
    \ta_2^n
    \cdot
    \frac{\tau^n}{n!}
    +
    (n!)^2
    \cdot
    \ta_2^n
    \cdot
    \left(
    \sum_{\l = 0}^{h(n)}
    \frac{(n - h(n))^{\ta_3 - 1}}{\Ga(\ta_3)}
    \cdot
    \frac{\tau^\l}{\l!}
    +
    \sum_{\l = h(n) + 1}^{n - 1}
    \frac{(n - \l)^{\ta_3 - 1}}{\Ga(1)}
    \cdot
    \frac{\tau^\l}{\l!}
    \right)
    \nonumber                     \\
     & = (n!)^2
    \cdot
    \ta_2^n
    \left[
      \frac{\tau^n}{n!}
      +
      \frac{n^{\ta_3 - 1}}{\Ga(\ta_3)}
      \cdot
      \left(
      1
      -
      \frac{h(n)}{n}
      \right)^{c - 1}
      \cdot
      \sum_{\l = 0}^{h(n)}
      \frac{\tau^\l}{\l!}
      +
      \sum_{\l = h(n) + 1}^{n - 1}
      \frac{1}{(n - \l)^{1 - \ta_3}}
      \cdot
      \frac{\tau^\l}{\l!}
      \right],
    \label{eq:Psi:n:2}            \\
    \frac{(1 + \frac{\ta_3}{n})^{\ta_3 - 1}}{\Ga(\ta_3)}
    \sum_{\l = 0}^n
    \frac{\tau^\l}{\l!}
     & \leq
    \frac{\Psi_n(\ta_1, \ta_2, \ta_3)}{(n!)^2 \ta_2^n \cdot n^{\ta_3 - 1}}
    \nonumber                     \\
     & \leq
    \frac{\tau^n \cdot n^{1 - \ta_3}}{n!}
    +
    \frac{\left(1 - \frac{h(n)}{n}\right)^{\ta_3 - 1}}{\Ga(\ta_3)}
    \sum_{\l = 0}^{h(n)}
    \frac{\tau^\l}{\l!}
    +
    \sum_{\l = h(n) + 1}^{n - 1}
    \frac{n^{1 - \ta_3}}{(n - \l)^{1 - \ta_3}}
    \cdot
    \frac{\tau^\l}{\l!}.
    \label{eq:Psi:n:3}
  \end{align}
  \rule{\linewidth}{0.1pt}
\end{figure*}

\begin{corollary}
  Given Assumption~\ref{ass:iid:matrices:1}, it holds that
  \begin{align*}
    \expvalbig{\perm(\matr{A})^2}
     & = \Psi_n(\mu_2, \mu_1^2, 1),                      \\
    \expvalbig{\perm_{\Bethe, 2}(\matr{A})^2}
     & = \Psi_n\left(\mu_2, \mu_1^2, \frac{1}{2}\right), \\
    \perm(\matr{1}_{n \times n})^2
     & = \Psi_n(1, 1, 1),                                \\
    \perm_{\Bethe, 2}(\matr{1}_{n \times n})^2
     & = \Psi_n\left( 1, 1, \frac{1}{2} \right).
  \end{align*}
\end{corollary}

\begin{proof}
  These expressions follow from expressing the quantities on the left-hand
  side in terms of Bell polynomials and using \Cref{lemma:psi-func-prop}.
\end{proof}

\IEEEtriggeratref{21}

\bibliographystyle{IEEEtran}
\bibliography{refs}

\end{document}